\documentclass[10pt]{amsart}
\usepackage{amssymb,amsmath, amsthm, amsfonts}
\usepackage{mathrsfs,color}
\usepackage{float}
\usepackage[all]{xypic}
\usepackage{alphalph}

\newcommand{\luk}{\L u\-ka\-si\-e\-w\-icz}
\usepackage{pgf,tikz}
\newtheorem{theorem}{Theorem}[section]
\newtheorem{lemma}[theorem]{Lemma}
\newtheorem{corollary}[theorem]{Corollary}
\newtheorem{proposition}[theorem]{Proposition}

\theoremstyle{definition}

\newtheorem{definition}[theorem]{Definition}
\newtheorem{remark}[theorem]{Remark}
\newtheorem{example}[theorem]{Example}

\newcommand{\scr}[1]{\mathscr {#1}}

\newcommand{\free}{{\bf F}}

\newcommand{\alg}[1]{{\textbf{\upshape #1}}}  %
\newcommand{\vv}[1]{\mathsf{#1}}
\newcommand{\cc}[1]{\mathcal{#1}}
\newcommand{\app}{\approx}
\newcommand{\bt}{{\bf t}}

\newcommand{\Con}[1]{\operatorname{Con}(\alg #1)}

\newcommand{\vuc}[2]{#1_1,\dots,#1_{#2}}
\newcommand{\Sym}{\mathsf{Sym}}
\newcommand{\Alg}{\mathsf{Alg}}

\newcommand{\sse}{\subseteq}
\newcommand{\Var}{{\rm Var}}

\newcommand{\Gen}{\mathscr{G}}
\newcommand{\Z}{\mathbb{Z}}

\usepackage[left=3cm,right=3cm,top=2.5cm,bottom=2.5cm]{geometry}

\begin{document}
\title[]{Generalization of terms via universal algebra}

\author{Tommaso Flaminio}
\address{IIIA -- Artificial Intelligence Research Institute,
 CSIC -- Spanish National Research Council. 
 Campus UAB s/n, Bellaterra 08193, Spain}
 \email{\texttt{tommaso@iiia.csic.es}}
\author{Sara Ugolini}
\address{IIIA -- Artificial Intelligence Research Institute,
 CSIC -- Spanish National Research Council. 
 Campus UAB s/n, Bellaterra 08193, Spain}
 \curraddr{DIISM, Universit\`a di Siena, via Roma 56, 53100 Siena, Italy}
 \email{\texttt{sara.ugolini@unisi.it}}

\date{}

\maketitle


\begin{abstract}
We provide a new foundational approach to the generalization of terms up to equational theories.  We interpret generalization problems  in a universal-algebraic setting making a key use of projective and exact algebras in the variety associated to the considered equational theory. We prove that the generality poset of a problem and its type (i.e., the cardinality of a complete set of least general solutions) can be studied in this algebraic setting. Moreover, we identify a class of varieties where the study of the generality poset can be fully reduced to the study of the congruence lattice of the 1-generated free algebra. We apply our results to varieties of algebras and to (algebraizable) logics. In particular we obtain several examples of unitary type: abelian groups; commutative monoids and commutative semigroups; all varieties whose 1-generated free algebra is trivial, e.g., lattices, semilattices, varieties without constants whose operations are idempotent; Boolean algebras, Kleene algebras, and G\"odel algebras, which are the equivalent algebraic semantics of, respectively, classical, 3-valued Kleene, and G\"odel-Dummett logic. Finally, we prove that the variety of MV-algebras, the equivalent algebraic semantics of \L ukasiewicz logic, has nullary type. 
\end{abstract}


\section{Introduction}
A term $s$ is  a {\em generalization} of a term $t$ if $t$ can be obtained from $s$ through a variable substitution. The problem of identifying common generalizations for two or more terms has been the focus of substantial research, initiated in a series of papers by Plotkin \cite{Plotkin}, Popplestone \cite{Popplestone}, and Reynolds \cite{Reynolds}, all collected in the same volume published in 1970.  The objective of these initial papers was to formalize an abstraction of the  process  of {\em inductive reasoning}. As a terminological aside, this abstraction  is named ``inductive generalization''  in \cite{Plotkin}, whilst more recent literature simply speaks of ``generalization processes'' (see e.g. \cite{CernaKutsia}). Here we conform to this latter usage. 
The main aim in this context is to find the solutions, i.e., the generalizing terms, that are {\em as close as possible} to the initial terms that define the problem. The cardinality of this set of ``best'' solutions is encoded in what is called the {\em generalization type}. 
In the literature,
the collection of methods and techniques  developed to  compute solutions to a generalization problem often goes under the name of {\em anti-unification}. This  terminology  suggests a connection between generalization and the arguably better-known {\em unification} problems, where one seeks common instantiations to pairs of given terms. 

The present paper provides a novel foundational approach to {\em equational} generalization, i.e., where terms are understood to be equivalent up to an equational theory. Our approach is based on universal algebra, which is a most natural environment to handle equational theories from the side of their classes of models, i.e., {\em varieties}. 

The extension to generalization up to equational theories has been  considered by several authors in theoretical computer science (see \cite{Baader,Burghardt,Pottier,Baumgartneretal,BN26}).  Most relevant results so far have been obtained with ad hoc techniques (see, e.g., results on semirings \cite{Cerna} and idempotent operations \cite{CernaKutsia}), and there is a growing interest in the general foundational aspects of it, as highlighted in the recent survey \cite{CernaKutsiasurvey}. 
We point out that our approach to equational theories is the usual one adopted in universal algebra, and as such it has a different flavor from the one that has been until now considered in the aforementioned  literature. In this work, we consider {\em arbitrary} equational theories over {\em fixed} languages, and terms are completely identified up to equational theories as elements of free algebras. 
In the literature, one often finds results about languages where one could have any number of operations, but they all satisfy the same equations such as associativity (A), commutativity (C), and their combination (AC). In this setting terms are thought of as being rewritten by means of such equations. This is a methodological difference more than substantial one, but it does shift the presentation and offers a different perspective.
 We also point out to the reader that some existing results assume the existence of infinitely many free constants; such cases can be studied also in our framework, since universal algebra does allow structures with infinite signatures. Methodologically speaking, our approach is semantical in nature as opposed to the syntactic methods found in most of the literature on the generalization type; other approaches with a semantical flavour are introduced in \cite{Baader,Kahl}, which use categorical methods as opposed to universal algebraic ones.

 The foundational approach proposed in the present paper is 
inspired by Ghilardi's algebraic setting for the study of equational unification problems \cite{G97}. This has been widely used in the literature since it allows the use of (universal) algebraic techniques \cite{AgUg1,AgUg2,Dzik08,GMU24,Slomcz12} and duality theoretic methods employing Priestley-like \cite{BC13,CM15,Cabrer16,Ghilardi04} and geometrical \cite{FU23,MS13,Ugolini} dualities. 
As we shall show also in the present paper, an advantage given by the universal algebraic perspective is that it brings to light the {\em invariants} of both the problems and the solutions: translating a problem to the algebraic setting allows to identify which problems share the same solutions, and which solutions are essentially equivalent up to the considered equational theory. 

In this manuscript we translate equational generalization  problems ({\em e-generaliza}\-{\em tion problems} from now on)  and their solutions to  homomorphisms between special classes of algebras, using in particular {\em projective} and {\em exact} algebras in the variety associated to the considered equational theory. Our first main result demonstrates how the generality  type can be studied in the fully algebraic setting (Theorem \ref{them:posetiso}, Corollary \ref{cor:type}). A key role in our analysis is played by 1-generated algebras in the considered variety. While being 1-generated is not a categorical notion, we prove that our results are preserved under categorical equivalences that preserve free algebras, called just {\em equivalences} in \cite{Mckenzie} (Theorem \ref{thm:equivalence}); moreover, categorical methods can and will be used to study e-generalization problems, an example of this being shown in the last section with the variety of Kleene algebras.

After developing the general theory, we take advantage of some basic universal-algebraic tools, and develop a methodology based on the study of the congruence lattice of the 1-generated free algebra in the considered variety; particularly, we identify a class of varieties where the study of the generality type can be fully reduced to the study of this congruence lattice (Theorem \ref{thm1ESP}). Finally, we identify a sufficient condition for problems in a variety to always have a best solution, i.e., {\em unitary type} (Corollary \ref{corFinal}).

With our methods we are able to tackle a variety of examples, coming both from algebra and {\em algebraizable} logics \cite{BP89,Font}. Particularly, we show that problems in the following varieties have unitary type: abelian groups, commutative semigroups and monoids (Example \ref{ex:typegroups}); all varieties whose 1-generated free algebra is trivial, e.g., lattices, semilattices, varieties without constants whose operations are idempotent (Example \ref{ex:typelattices}); Boolean algebras, Kleene algebras, and G\"odel algebras, which are the equivalent algebraic semantics of, respectively, classical, 3-valued Kleene, and G\"odel-Dummett logic (Section \ref{section:logic}). Finally, we provide an interesting example of nullary type coming from logic: the variety of MV-algebras, which is the equivalent algebraic semantics of \L ukasiewicz infinite-valued logic. This result is obtained using the well-known Marra-Spada duality between finitely presented MV-algebras and a category of rational polyhedra \cite{MS13}. 

The paper is organized as follows: after a section of preliminaries,  in Section \ref{section:algpresentation} we develop the general theory of algebraic e-generalization problems; in Section \ref{section:congruence} we  study the e-generalization type via the congruence lattice of the 1-generated free algebra in the considered variety; lastly, in Section \ref{section:logic} we apply our results to algebraizable logics.

\section{Preliminaries}
\subsection{Universal algebraic preliminaries: projectivity and exactness}\label{sec2}
For all the unexplained notions of universal algebra we refer to \cite{MMT2018}. 
 If $\rho$ is a type {\em signature} of a class of algebras, a {\em $\rho$-equation} is a pair $(p,q)$ of $\rho$-terms  that we write suggestively as $p \app q$.  A set $\Sigma$ of equations over the same signature will be said to be an {\em equational theory}. In what follows we  usually omit the ignature, and just consider equations assuming a fixed one and, whenever we consider a class of algebras $\vv K$, we assume that the algebras in $\vv K$ have the same signature. 
 
 We denote by $\alg T_\rho(X)$ the absolutely free algebra of $\rho$-terms over a set $X$ of generators, and by $\alg F_\vv K(X)$ the free algebra for a class $\vv K$ over $X$. If $X = \{x\}$, we will often write $\alg F_\vv K(x)$ for $\alg F_\vv K(\{x\})$.

 Given any set of variables $X$, we call {\em assignment} of $X$ into an algebra $\alg A$ a function $h$ mapping each variable $x \in X$ to an element of $\alg A$. If $\rho$ is the signature of ${\bf A}$, each assignment  extends (uniquely) to a homomorphism, which we also call $h$, from $\alg T_\rho(X)$ to $\alg A$.
An algebra $\alg A$ satisfies an equation $p \app q$ with an assignment $h$, and we write $\alg A, h \vDash p \approx q$, if $h(p) = h(q)$ in $\alg A$; $\alg A$ is a {\em model} of $p \app q$, and we write $\alg A \vDash p \approx q$, if for all assignments $h$ in $\alg A$, $\alg A, h \vDash p \approx q$. If $\Sigma$ is an equational theory then $\alg A$ is said to be a {\em model} of $\Sigma$, and we write $\alg A \vDash \Sigma$, if $\alg A \vDash \sigma$ for all $\sigma \in \Sigma$. Finally, a class of algebras $\vv K$ is a {\em class of models} of $\Sigma$, and we write $\vv K\models \Sigma$, if all algebras in $\vv K$ are a model of $\Sigma$.

 A class of algebras is a {\em variety} if it is the class of all models of an equational theory. Equivalently, varieties are those classes of algebras closed under the algebraic operations of taking homomorphic images, subalgebras, and direct products. As a consequence, varieties contain all of their free algebras.
 
  Fixed a variety $\vv V$, by a {\em substitution} $\sigma$ we mean an endomorphism of $\free_{\vv V}(\omega)$, the free algebra in $\vv V$ over a denumerable set of variables indexed by $\omega$, the set of cardinality $\aleph_0$. Note that a substitution is fully determined by the values it assigns to each generating variable. Usually we consider substitutions that are different from the identity map only on a subset of variables $X$, mapping them to terms written over variables over a set $Y$; we highlight this fact by considering such a substitution as a homomorphism from  $\free_{\vv V}(X)$ to  $\free_{\vv V}(Y)$. 
 
We now introduce projective algebras, which play a key role in our investigation.
\begin{definition}
	Given a class of algebras $\vv K$, an algebra $\alg P \in \vv K$ is {\em projective in $\vv K$} if for all $\alg A,\alg B \in \vv K$, homomorphism $h:\alg P \to \alg B$, and surjective homomorphism $g: \alg A\to  \alg B$, there is a homomorphism $f: \alg P \to  \alg A$ such that $h=g\circ f$.
\end{definition}
While in general it is difficult to determine the algebras that are projective in a class $\vv K$, whenever $\vv K$ contains all of its free algebras (hence in particular for $\vv K$ being a variety) projectivity admits a more intuitive characterization. Let us call an algebra $\alg A$ a {\em retract} of an algebra $\alg F$ if there are homomorphisms $i: \alg A \to \alg F, j: \alg F \to \alg A$ such that $j \circ i = id_{\alg A}$ (and then necessarily $i$ is injective and $j$ is surjective). The following theorem was first proved by Whitman for lattices \cite{Whitman1941}, but it is well-known to hold for any class of algebras $\vv K$;  its proof is an exercise in universal algebra.
\begin{theorem}\label{prop:proj-retract}
	If a class of algebras $\vv K$ contains all its free algebras, an algebra $\alg P$ is projective in $\vv K$ if and only if it is a retract of a free algebra in $\vv K$.
\end{theorem}

Thus, free algebras in a class $\vv K$ are clearly projective in $\vv K$. 
Since the above characterization of projective algebras in a variety will be extensively used along the paper, it is convenient to  graphically represent it as in the following diagram:
$$
\xymatrix{
 \alg P  \ar@(ul,dl)[]_{j \circ i = {\rm id}_{\alg P}} \ar@<0.6ex>[rr]^-i &&  \alg F_{\vv V}(Y) \ar@<0.6ex>[ll]^-j
}
$$
Given the above representation, and to highlight the role of the two mappings $i$ and $j$ that determine the retraction, we will henceforth say that the projective algebra $\alg P$ is an {\em $(i,j)$-retract of }$\alg F_{\vv V}(Y)$.

We now introduce {\em exactness}, a  weaker notion than projectivity relevant to this work. 
\begin{definition}
	Consider a variety $\vv V$; an algebra $\alg E \in \vv V$ is called {\em exact in $\vv V$} if it is isomorphic to a finitely generated subalgebra of some (finitely generated) free algebra.
\end{definition}
Evocatively, if we consider an exact algebra that is (isomorphic to) a subalgebra of a free algebra $\free_{\vv V}(X)$ generated by a term $t$, we write it as $\alg E_{\free_{\vv V}(X)}(t)$; whenever there is no danger of confusion, we write simply $\alg E(t)$.

Similarly to above, to highlight the role of an embedding $e$ that maps an exact algebra $\alg E$ into a free algebra $\free_{\vv V}(X)$, we will henceforth say that ${\bf E}$ is an {\em $e$-subalgebra of }$\free_{\vv V}(X)$.

Let us further observe that another way of thinking about exact algebras is to see them as images of substitutions from a finitely generated free algebra to another (and in fact, {\em exact} formulas are those that represent the theory of a substitution as in \cite{DJ1,DJ2,DJ3}).
By the characterization in Theorem \ref{prop:proj-retract}, every finitely generated projective algebra is exact, while the converse is not true in general. 

Projective and exact algebras in varieties, due to their close relation with free algebras, have shown to serve a key role in the algebraic study of problems coming from logic, particularly, in unification problems and the related notion of structural completeness (see, e.g., \cite{AgUg24, CM15a, G97}). This work shows their importance in the study of e-generalization problems as well.

In order to study the above notions, we often make use of congruences and the associated quotients. Generally, given an algebra $\alg A$, we write $\Con A$ for its set of congruences. Given a set $\Phi \sse \mathscr{P}(A \times A)$ of pairs of elements of $\alg A$, one can consider the congruence generated by $\Phi$ in $\alg A$, ${\rm Con}_{\alg A}(\Phi)$; if $\Phi = \{(x,y)\}$, and if there is no danger for confusion, we simply write $\langle x,y\rangle$ for ${\rm Con}_{\alg A}(\{(x,y)\})$. Given any $\theta \in  \Con A$, we write $n_\theta$ for the {\em natural epimorphism} $n_\theta: \alg A \to \alg A/ \theta$ mapping each $a \in A$ to its congruence class $a/\theta$.
Moreover, given any homomorphism $h: \alg A\to\alg B$ one can consider the congruence given by its {\em kernel}:
$$
\ker(h)=\{(a,a') \in A\times A: h(a)=h(a')\} \in \Con A.$$

We recall that varieties are closed under taking quotients (equivalently, homomorphic images) and every algebra $\alg A$ in a variety $\vv V$ is (isomorphic to) a quotient of some free algebra, say $\alg A \cong \free_{\vv V}(X)/\theta$ for some set of variables  $X$ and $\theta\in \Con{\free_{\vv V}(X)}$. We can then transfer the notions above from the algebras to the corresponding congruences.
\begin{definition}\label{congProj}
Let $\mathsf{V}$ be a variety. We say that a congruence $\theta$ of a free algebra $\free_{\mathsf{V}}(X)$ is {\em projective} (or {\em exact}) if  $\free_{\mathsf{V}}(X)/\theta$ is projective (or exact) in $\mathsf{V}$.
\end{definition}

The relation between congruences, homomorphic images, and quotients of algebras is a basic tool from universal algebra; the main facts are expressed in the so-called {\em Isomorphism Theorems} \cite[\S 4.2]{MMT2018}. We recall the First Isomorphism Theorem (often called just Homomorphism Theorem) and Second Isomorphism Theorem here for the reader's convenience, since they will be widely used in  proofs.
\begin{theorem}\label{thm:secondiso}
\begin{enumerate}
	\item[{\bf (1)}] (First Isomorphism Theorem) Let $h : \alg A \to \alg B$ be a homomorphism, let $\theta = \ker(h)$, and consider the natural epimorphism $n_\theta$ from $\alg A$ to $\alg A/ \theta$; the unique function $f: \alg A/ \theta \to \alg B$ such that $f \circ n_\theta = h$ is an embedding. Moreover, $f$ is an isomorphism if and only if $h$ is onto $\alg B$. 
	\item[{\bf (2)}] (Second Isomorphism Theorem) Let $f: \alg A \to \alg B$ and $g: \alg A \to \alg C$ be homomorphisms such that $\ker(f) \sse \ker(g)$ and $f$ is onto $\alg B$. Then there is a unique homomorphism $h: \alg B \to \alg  C$ such that $g= h \circ f$.
Moreover, $h$ is an embedding if and only if $\ker(f) =\ker(g)$.

\end{enumerate}
\end{theorem}
As a consequence of the Second Isomorphism Theorem one can actually see that given any variety $\vv V$, any assignment $h: \alg T_\rho(X) \to \alg A \in \vv V$ induces a homomorphism $\bar h: \free_{\vv V}(X) \to \alg A$; precisely, if $\bar x$ is the representative in $\free_{\vv V}(X)$ of a variable $x \in X$, $\bar h(\bar x) = h(x)$. Therefore, the validity of equations in a variety can be equivalently checked by considering assignments with domain over a free algebra $\free_{\vv V}(X)$ for some $X$ (instead of the absolutely free algebra $\alg T_\rho(X)$) to algebras in $\vv V$. We will use this fact without further notice, and just consider terms as elements of free algebras in the considered variety. Lastly, given a term $t$, we write $\Var(t)$ for the variables appearing in $t$.

\subsection{E-generalization problems}
Observe that given any equational theory $\Sigma$, we can  consider the associated class of models $\vv V_{\Sigma}$, which is a variety of algebras. Similarly, as recalled in the previous subsection, any variety $\vv V$ is exactly the class of models of some equational theory $\Sigma_{\vv V}$: $\Sigma = \Sigma_{\vv V_{\Sigma}}$ and $\vv V = \vv V_{\Sigma_{\vv V}}$. Moreover, given two terms $t, u$ over the same signature, $t = u$ up to $\Sigma$ if and only if $\vv V \models t \approx u$. Therefore one can use equational theories and varieties interchangeably; we will henceforth focus on varieties with no loss of generality. 

The following definition of an e-generalization problem corresponds to the usual one used in the literature, rephrased in the (equivalent) context of varieties and their free algebras; we only observe that while in the literature e-generalization problems are often considered to be a pair of terms, we here consider the more general case of allowing a finite multiset of terms of any (finite) cardinality. The choice of considering {\em multi}sets of terms is to allow the  terms to be repeated up to the considered equational theory.
 
 \begin{definition}
	A {\em symbolic e-generalization problem} for a variety $\vv V$ is a finite multiset ${\bf t}$ of terms $\vuc{t}{m} \in \alg F_{\vv V}(X)$ for some finite set of variables $X$. A {\em solution} (or {\em generalizer}) is a term $s \in \alg F_{\vv V}(Y)$, with $Y = \Var(s)$, for which there exist substitutions $\vuc{\sigma}{m}$ such that $\vv V\models \sigma_k(s) \app t_k$ for all $k = 1, \ldots, m$. In this case we say that $s$ is {\em witnessed} (or {\em testified}) by  $\sigma_1,\ldots, \sigma_m$. 
\end{definition}

As already noticed by Plotkin for syntactic generalization \cite{Plotkin}, any symbolic generalization problem $\vuc{t}{m}$ always has a solution: 
a fresh variable $z$, testified by the obvious substitutions $\sigma_k(z) = t_k$ for $k = 1, \ldots, m$.  
This  is the {\em most general solution} for $\vuc{t}{m}$, in the sense that every other solution can be obtained from it by further substitution.
In this context the interesting solutions are the {\em least general} ones, that are as {\em close} as possible to the initial terms representing the problem. Let us make this precise.

Consider two terms over the same language $s$ and $u$; we say that $s$ is {\em less general} than $u$, and write 
\begin{equation}
	s \preceq u \mbox{, iff  there exists a substitution } \sigma \mbox{ such that } \sigma(u) = s.
\end{equation}
Let us then fix a problem $\bt \subseteq \alg F_{\vv V}(X)$  and let  $\mathscr{S}(\bt)$ be the set of its solutions. It is easy to see that $\preceq$ is a preorder on $\mathscr{S}(\bt)$.  With a slight abuse of notation we  denote by 
$(\mathscr{S}(\bt), \preceq)$ its associated poset of equally general solutions, that we call the {\em generality poset} of $\bt$. 

For a problem $\bt$, one is usually interested in determining a set of  minimal solutions in its generality poset, whether it exists, and its cardinality. This information is encoded in the {\em e-generalization type}, which intuitively gives the cardinality of a class of best (i.e., minimal) solutions. Let us be more precise. 

Given a poset $(P, \leq)$, we call a {\em minimal (resp. maximal) complete set} a subset $M \sse P$ such that:
\begin{enumerate}
	\item for every $p \in P$, there is an element of $a \in M$ such that $a \leq p$ (or $p \leq a$);
	\item no two distinct elements of $M$ are comparable with respect to the order.
\end{enumerate}
Given a symbolic e-generalization problem $\bt$, its {\em symbolic e-generalization type} is:
\begin{itemize}
	\item unitary: if $(\mathscr{S}(\bt), \preceq)$ has a minimal complete set of generalizers (or {\em mcsg}) of cardinality 1;
	\item finitary:  if $(\mathscr{S}(\bt), \preceq)$ has an mcsg of finite (greater than $1$) cardinality;
	\item infinitary: if $(\mathscr{S}(\bt), \preceq)$ has an mcsg of infinite (not finite) cardinality;
	\item nullary: if $(\mathscr{S}(\bt), \preceq)$ has no mcsg.
\end{itemize}
Given a variety $\vv V$, its symbolic e-generalization type is the worst possible type occurring among all its e-generalization problems, the best-to-worst order being: unitary $>$ finitary $>$ infinitary $>$ nullary.

In order to recover some of the results present in the literature, we also introduce the notion of {\em e-generalization $m$-type} for $\vv V$, for $m \in \mathbb{N} - \{0\}$, which is the worst possible type occurring among all the e-generalization problems for $\vv V$ where the cardinality of the multiset of terms is at most $m$. To conform to this notation, we will also call the e-generalization type of $\vv V$ the {\em e-generalization $\omega$-type of $\vv V$}.

Notice that if $m \leq n$, the e-generalization $m$-type of $\vv V$ is better or equal to the $n$-type of $\vv V$. Moreover: 
\begin{proposition}
	Let $\vv V$ be a variety. If $\vv V$ has unitary e-generalization $2$-type, then $\vv V$ has unitary e-generalization type.
\end{proposition}
\begin{proof}
	Suppose $\vv V$ has unitary e-generalization $2$-type, and consider a problem $\bt = \{t_1, \ldots, t_m\}$. If $m = 1$, it is clear that the e-generalization type of $\bt$ is unitary, since we can consider the completely equivalent problem $\{t_1, t_1\}$. Let now $m > 2$. We can find a least general solution to $\bt$ with the following procedure:
	\begin{enumerate}
		\item Partition the terms in $\bt$ into the 2-element sets $$\bt_{1,2}:=\{t_1, t_2\}, \ldots, \bt_{m-1,m}:=\{t_{m-1}, t_m\}$$ if $m$ is even, or, if $m$ is odd, consider the last multiset to be $\bt_{m,m} := \{t_{m}, t_m\}$. Let $s_1, \ldots, s_{\lceil m/2 \rceil}$ (where $\lceil n \rceil$ is the least integer greater than or equal to $n$) be the set of least general solutions to the above sets of problems, which exist since $\vv V$ has unitary e-generalization $2$-type.
		\item Repeat the above process on the set of solutions, until there is only one solution, call it $s$. 
	\end{enumerate}
Note that the procedure in point (1) is done $\lceil m/2 \rceil$ times. While it is clear that $s$ is a solution for $\bt$, we argue that $s$ is minimal, i.e., it is  a least general solution for $\bt$. Indeed, consider another solution $s'$. Then $s'$ is also a solution for all the pairs in the first step of the procedure, $\bt_{1,2}, \ldots, \bt_{m-1,m}$ (or $\bt_{m,m}$). Thus $s_1, \ldots, s_{\lceil m/2 \rceil}$ are all less general than $s'$, i.e., there are substitutions $\sigma_1, \ldots, \sigma_{\lceil m/2 \rceil}$ such that $\sigma_k(s') = s_k$ for $k = 1, \ldots, \lceil m/2 \rceil$. This means that $s'$ is also a solution for all the pairs considered in the second iteration of the procedure. Applying the same reasoning at each repetition, we finally get that $s'$ is a solution also to the last pair of terms considered, whose least general generalizer is $s$. Thus, $s \preceq s'$ and the proof is complete. 
\end{proof}
Note that the same reasoning cannot be applied if $\vv V$ has, for instance, finitary e-generalization $2$-type to conclude that it has finitary e-generalization type, since the procedure described in the proof above might not terminate. 
We leave the question whether the e-generalization type coincides with the 2-type for every variety open.

One last observation before we conclude the preliminary section.
\begin{remark}\label{rem:one}
	For the sake of the reader more familiar with unification problems, it is worth stressing that both the solutions to e-generalization problems and their order have  a different behavior than those of unification theory. Indeed, the solutions to a unification problem are {\em substitutions} that unify (possibly up to some equational theory) all the pairs of terms considered in the problem. Here solutions are identified as terms, while substitutions only serve as witnesses. Importantly, note that one could have two different solutions $s, s'$ to some e-generalization problem ${\bf t} = \{t_1, \ldots, t_m\}$, witnessed by the same substitutions $\sigma_1, \ldots, \sigma_m$, i.e., $\sigma_k(s) = \sigma_k(s') = t_k$ for $k = 1, \ldots, m$, but $s$ and $s'$ could be incomparable in the generality order. Indeed,  there might be no substitution mapping $s$ to $s'$ or vice versa. This observation is crucial to gain a deeper understanding of the algebraic approach we put forward in this work.
	\end{remark}
\section{An algebraic presentation of e-generalization}\label{section:algpresentation}
In this section we present a novel universal-algebraic approach to e-generalization problems. 
The  starting intuition is that, given a symbolic problem $t_1, \ldots, t_m$, each term defines a 1-generated subalgebra of some free algebra, i.e., a 1-generated exact algebra $\alg E(t_k)$. As we hinted at in Remark \ref{rem:one}, it is necessary for the study of the generalization order to fix the specific term $t_k$, for each $k = 1, \ldots, m$, and not just the exact algebra it generates (which could have different generators).
The idea that makes our translation work is to see the terms $t_1, \ldots, t_m$ as a {\em single element} $(t_1, \ldots, t_m)$ of the direct product of the exact algebras $\alg E(t_k)$; it will be convenient to represent this tuple as the image of a fresh variable $z$ via some homomorphism $h$ on the 1-generated algebra in the associated variety, $\free_{\vv V}(z)$. Then, each term $t_k$ can be recovered as $p_k \circ h(z)$, where $p_k: \prod_{k= 1}^m \alg E(t_k) \to \alg E(t_k)$ is the $k$-th projection. 

Hence, the tuple of terms can be substituted by a tuple of surjective homomorphisms $h_k: \free_{\vv V}(z) \to \alg E(t_k)$, or, equivalently and more compactly, by a homomorphism $h: \free_{\vv V}(z) \longrightarrow \prod_{k= 1}^m \alg E(t_k)$ such that $p_k \circ h$ is surjective on $\alg E(t_k)$.
The next definition translates this intuition in more abstract algebraic terms.

\begin{definition}\label{def:functunif}
	We call an {\em  algebraic e-generalization problem} for a variety $\vv V$  a homomorphism $h: \free_{\vv V}(z) \longrightarrow \prod_{k= 1}^m \alg E_k$ for some $m \geq 1$, where each $\alg E_k$ is a 1-generated exact algebra in $\vv V$, and $p_k \circ h$ is surjective on $\alg E_k$:
		 $$
\xymatrix{
\free_{\vv V}(z)\ar@{->>}[drr]_-{p_k \circ h}\ar[rr]^-h&& \displaystyle\prod_{k = 1}^m \alg E_k \ar@{->>}[d]^{p_k}\\
&& \alg E_k  \\
}
$$
	
	A {\em solution} (or {\em generalizer}) for $h$ is any homomorphism $g: \free_{\vv V}(z) \longrightarrow \alg P$, where $\alg P$ is finitely generated and projective in $\vv V$, for which there exists a homomorphism $f: \alg P \longrightarrow \prod_{k = 1}^m \alg E_k$ such that $f \circ g = h$, as illustrated in the following diagram:
	 $$
\xymatrix{
\free_{\vv V}(z)\ar[d]_-g\ar[rr]^-h&& \displaystyle\prod_{k = 1}^m \alg E_k\\
 \alg P   \ar[urr]_-{f} & \\
}
$$
We say that $f$ {\em witnesses} or {\em testifies} the solution $g$.
	\end{definition}
 The reader shall notice that, as a consequence of the definition, the algebras $\alg E_k$ in the definition of an algebraic problem are (1-)generated by the element $p_k \circ h(z)$.

 \begin{remark}\label{remark:equivalentpresentation}
 	We observe that an equivalent presentation of an algebraic problem would be to consider a tuple $\{h_k\}_{k = 1}^m$ where each $h_k: \free_{\vv V}(z) \to \alg E_k$ is a surjective homomorphism to a 1-generated exact algebra. With respect to this presentation, a solution for $\{h_k\}_{k = 1}^m$ is a homomorphism $g: \free_{\vv V}(z) \longrightarrow \alg P$, where $\alg P$ is finitely generated and projective in $\vv V$, for which there exists a family of homomorphisms  $\{f_k\}_{k = 1}^m$ with $f_k: \alg P \to \alg E_k$ such that $f_k \circ g = h_k$. We picked the above more compact presentation, as it will make clearer some of the general results shown in the next section. Nonetheless, it is useful to keep in mind this different but equivalent perspective. We will make use of it in one of the applications, namely the case of \L ukasiewicz logic, in Section \ref{sec:luk}. 
 \end{remark}

	One can also interpret the generality order among solutions in this alternative setting. 
Consider two homomorphisms with the same domain, $g: \alg A \to \alg B$ and $g': \alg A \to \alg B'$. We say that $g$ is {\em less general} than $g'$ and we write
\begin{equation}\label{eqOrderAlgSol}
	g \sqsubseteq g' \mbox{ if and only if there exists a homomorphism } l: \alg B' \to \alg B \mbox{ such that } l \circ g' = g,
\end{equation}
as in the following diagram:
$$
\xymatrix{
\alg A \ar[d]_-{g'}\ar[drr]^-{g}&&  \\
 \alg B' \ar[rr]_-l& & \alg B
 }
$$

Let us fix an algebraic problem $h: \free_{\vv V}(z) \to \prod_{k = 1}^m \alg E_k$. 
Given two generalizers $g: \free_{\vv V}(z) \to \alg P$, $g': \free_{\vv V}(z) \to \alg P'$, we hence have that $g \sqsubseteq g'$ if and only if there exists $f: \alg P' \to \alg P$  such that $f \circ g' = g$. The relation $\sqsubseteq$ is easily checked to be a preorder on the set of generalizers for $h$. We write 
\begin{equation}
	(\mathscr{A}(h), \sqsubseteq).
\end{equation}
 for the corresponding poset of equally general generalizers. 
 
The {\em algebraic e-generalization type} of a problem is then given in complete analogy with the symbolic case, by checking the cardinality of a minimal complete set of solutions; similarly, one can define the algebraic e-generalization $m$-type of a variety as the worst possible type of its problems of the kind $h: \free_{\vv V}(z) \to \prod_{k = 1}^m \alg E_k$ for any $\omega > m \geq 1$. Recall that we call the {\em $\omega$-type} also just {\em type}, and this corresponds to considering all possible problems of any (finite) cardinality.

We will prove that the algebraic and symbolic e-generalization type for a variety coincide. Let us start by
showing how to translate back and forth between symbolic and  algebraic problems and solutions. 
\begin{definition}\label{def:symbalg}
	Let $\bt=\{t_1,\ldots, t_m\} \subseteq \alg F_{\vv V}(X)$ be a symbolic e-generalization problem for a variety $\vv V$, and $s \in \alg F_{\vv V}(Y)$ be a solution. Let us define $\Alg (\bt)$ and $\Alg(s)$ as the (unique) homomorphisms extending the following assignments: 
	\begin{eqnarray*}
		\Alg (\bt): z\in \free_{\vv V}(z) &\longmapsto& (t_1,\ldots, t_m)\in \prod_{k = 1}^m \alg E(t_k),\\
	 \Alg(s) :  z\in \free_{\vv V}(z) &\longmapsto& s\in \free_{\vv V}(Y).
	\end{eqnarray*}
\end{definition}
For every symbolic e-generalization problem $\bt$, $\Alg(\bt)$  is an algebraic e-generalization problem directly by Definition \ref{def:functunif}.  In particular, notice that for all $k=1,\ldots, m$, $p_k\circ \Alg(\bt):z\mapsto t_k$; $z$ generates $\free_{\vv V}(z)$ and $t_k$ generates $\alg E(t_k)$ and hence $p_k\circ \Alg(\bt)$ is surjective.
The next lemma proves that, if $s$ is a solution for $\bt$, $\Alg(s)$ is a solution to $\Alg(\bt)$. 
\begin{lemma}\label{lemma:algsol}
Let $\bt \sse \free_{\vv V}(X)$ be a symbolic e-generalization problem with a solution $s \in \alg F_{\vv V}(Y)$. Then
	$\Alg(s) $ is a solution of $\Alg (\bt)$.
\end{lemma}
\begin{proof}
Let $\bt = \{t_1, \ldots, t_m\}\sse \free_{\vv V}(X)$, and $\Alg(\bt):\free_{\vv V}(z)\to \prod_{k=1}^m \alg E(t_k)$ be the algebraic e-generalization problem as in Definition \ref{def:symbalg}. Note that $\Alg(s)$ has as codomain $\alg F_{\vv V}(Y)$, which is free and hence projective in $\vv V$ (Theorem \ref{prop:proj-retract}).

Let us assume that the solution $s$ is testified by some substitutions $\sigma_1,\ldots, \sigma_m$. Then, the map $f: \alg F_{\vv V}(Y)\to \prod_{k=1}^m \alg E(t_k)$ defined by 
$$
f(u)=(\vuc{\sigma}{m})(u) = (\sigma_1(u), \ldots, \sigma_m(u))
$$
for all $u\in \alg F_{\vv V}(Y)$ testifies that $\Alg(s) $ is a solution of $\Alg (\bt)$. Indeed, $f \circ \Alg(s) = \Alg(\bt)$ given that they coincide on the generator of $\free_{\vv V}(z)$:
$$f \circ \Alg(s)(z) = f(s) = (\sigma_1(s), \ldots, \sigma_m(s)) = (t_1, \ldots, t_m) = \Alg(\bt)(z).\qedhere$$
\end{proof}

For the converse translation, recall that we say that $\alg E$ is an $e$-subalgebra of $\alg F$ if $e: \alg E \to \alg F$ is an embedding, and that we write $p_k: \prod_{k = 1}^m \alg E_k \to \alg E_k$ for the $k$-th projection.

\begin{definition}\label{def:algsymb}
Let $h: \free_{\vv V}(z) \to \prod_{k = 1}^m \alg E_k$ be an algebraic e-generalization problem, where each $\alg E_k$  is  an $e_k$-subalgebra of $\free_{\vv V}(X_k)$, and consider a solution  $g: \free_{\vv V}(z) \to \alg P$, where $\alg P$ is an $(i,j)$-retract of $\alg F_{\vv V}(Y)$. We define:
\begin{eqnarray*}
	\Sym(h) &=& \{t_1,\ldots, t_m\},\; \mbox{ with } t_k = e_k \circ p_k \circ h(z) \mbox{ for } k = 1, \ldots, m;\\
	\Sym(g) &=& (i\circ g)(z)\in \free_{\vv V}(Y).
\end{eqnarray*}
\end{definition}

It is immediate that $\Sym(h)$ is a symbolic e-generalization problem for any algebraic e-generalization problem $h$; note that $\Sym(h) = \{t_1,\ldots, t_m\}\sse \free_{\vv V}(X)$ where $X = \Var\{t_1,\ldots, t_m\}$. We now demonstrate that $\Sym(g)$ is a solution for $\Sym(h)$.

\begin{lemma}\label{lemma:symsol}
Let $h: \free_{\vv V}(z) \to \prod_{k = 1}^m \alg E_k$ be an algebraic e-generalization problem, with $\alg E_k$ an $e_k$-subalgebra of $\free_{\vv V}(X_k)$,  
and let $g: \free_{\vv V}(z) \to \alg P$ be a solution of $h$, with $\alg P$ an $(i,j)$-retract of $\alg F_{\vv V}(Y)$. 
Then,  
	$\Sym(g)$ is a solution of $\Sym(h)$. 
\end{lemma}

\begin{proof}
Let $f: \alg P \to  \prod_{k = 1}^m \alg E_k$ be a witness for $g$ being a solution of $h$. 
In order to show that $\Sym(g)$ is a solution of $\Sym(h)$, consider the maps $\sigma_k=e_k \circ p_k \circ f \circ j$ as in the diagram below
 $$
\xymatrix{
\free_{\vv V}(z)\ar[d]_-g\ar[r]^-h& \displaystyle\prod_{k = 1}^m \alg E_k \ar[r]^-{e_k \circ p_k}& \free_{\vv V}(X_k)\\
 \alg P \ar@<-0.6ex>[d]_-i   \ar[ur]_-{f} && \\
 \alg F_{\vv V}(Y)  \ar@<-0.6ex>[u]_-j  \ar[uurr]_-{\sigma_k}& &
}
$$
	We need to show that $\sigma_k(\Sym(g)) = t_k$, i.e., that $\sigma_k(i(g(z))) = e_k \circ p_k \circ h(z)$ for all $k = 1, \ldots, m$, and indeed:
	$$\sigma_k(i(g(z))) = e_k \circ p_k \circ f \circ j \circ i \circ g(z) = e_k \circ p_k\circ f \circ g(z) = e_k \circ p_k \circ h(z),$$
	where in particular the second equality is given by the fact that $j \circ i = id_{\alg P}$ and the third follows from $f \circ g(z) =h(z)$.
\end{proof}
Note that $\Sym(g)$ depends on the pair $(i,j)$ witnessing the fact that $\alg P$ is a retract  of $\alg F_{\vv V}(Y)$. 
However, the choice of maps comes at no cost with respect to the study of the generality poset. Indeed, for the generalizers,	 consider $g: \alg F_{\vv V}(z)\to \alg P$ and let $(i,j)$ and $(i',j')$ testify that $\alg P$ is a retract of, respectively, $\free_{\vv V}(Y)$ and $\free_{\vv V}(Y')$. Then one would obtain two distinct symbolic solutions, $i(g(z))$ and $i'(g(z))$; the following compositions define substitutions mapping one term to the other: 
	$$i' \circ j: i(g(z)) \mapsto i'(g(z)); \quad i \circ j': i'(g(z)) \mapsto i(g(z)).$$
	Thus, the two solutions are equally general.
	Moreover, note that the definition of $\Sym(h)$ depends on the choice of the embeddings $e_k$; nonetheless, this does not have an effect on the study of the generality poset of solutions, as Theorem \ref{them:posetiso}(2) clarifies.
	
We are now ready to prove that given a symbolic problem $\bt \subseteq \alg F_{\vv V}(X)$, its poset of symbolic solutions is isomorphic to the poset of algebraic solutions for 
$\Alg(\bt)$. 
\begin{theorem}\label{them:posetiso} The following hold:
\begin{enumerate}
	\item Given a symbolic  problem $\bt$, its poset of solutions $(\mathscr{S}(\bt), \preceq)$ is isomorphic to the poset of algebraic solutions of $\Alg(\bt)$, $(\mathscr{A}(\Alg(\bt)), \sqsubseteq)$.
	\item Given an algebraic problem $h : \free_{\vv V}(z)\to\prod_{k=1}^m \alg E_k$, and any  choice of embeddings $e_k: \alg E_k \to \alg F_{\vv V}(X_k)$, the poset of solutions $(\mathscr{A}(h), \sqsubseteq)$ is isomorphic to the poset of symbolic solution of $\Sym(h)$, $(\mathscr{S}(\Sym(h)), \preceq)$.
\end{enumerate}
\end{theorem}
\begin{proof}
As for (1), let $\bt = \{t_1, \ldots, t_m\} \sse \free_{\vv V}(X)$. We first observe that given two symbolic solutions $s \in \alg F_{\vv V}(Y), s'  \in \alg F_{\vv V}(Y')$, $$s \preceq s' \mbox{ iff } \Alg(s) \sqsubseteq \Alg(s').$$
	Indeed, $s \preceq s'$ if and only if there exist $\sigma: \alg F_{\vv V}(Y') \to \alg F_{\vv V}(Y)$ such that $\sigma(s') = s$ if and only if there exists a  homomorphism $\sigma$ from $\alg F_{\vv V}(Y')$ to $\alg F_{\vv V}(Y)$ such that $\sigma\circ\Alg(s')=\Alg(s)$, if and only if $\Alg(s) \sqsubseteq \Alg(s')$. 
From here it immediately follows that $\Alg$ can be regarded as a (well-defined) map from $(\mathscr{S}(\bt), \preceq)$ to $(\mathscr{A}(\Alg(\bt)), \sqsubseteq)$ that is injective and order-preserving in both directions.
	
	We will now demonstrate that $\Alg$ is surjective. Consider an algebraic solution $g: \free_{\vv V}(z) \to \alg P \in \mathscr{A}(\Alg(\bt))$. Then there exists a witnessing  homomorphism $f: \alg P \to  \prod_{k = 1}^m \alg E(t_k)$ such that $f\circ g=\Alg(\bt)$. Since $\alg P$ is projective, there is a free algebra $\alg F_{\vv V}(Z)$ and homomorphisms $j: \alg F_{\vv V}(Z) \to \alg P, i: \alg P \to \alg F_{\vv V}(Z)$ such that $j \circ i = id_{\alg P}$. Consider then $s = (i\circ g)(z) \in \alg F_{\vv V}(Z)$; we show that $s$ is a solution for the symbolic problem, witnessed by the substitutions $\sigma_k = p_k \circ f \circ j$ for $k = 1, \ldots, m$ where once again $p_k$ is the $k$-th projection. Indeed, 
	$$
	\begin{array}{lll}
	\sigma_k(s) &=& \sigma_k\circ i\circ g(z) \\
	&= &p_k\circ f \circ j \circ i \circ g(z)\\
	& = &p_k\circ f \circ g(z) \\
	&=&p_k (\Alg(\bt)(z))\\
	& =& p_k(\vuc{t}{m}) \\
	&=& t_k
	\end{array}
	$$
 and  in the third equality we use that $j \circ i = id_{\alg P}$. Thus $s = i \circ g(z) \in \mathscr{S}(\bt)$.  Finally, note that $\Alg(s): \alg F_{\vv V}(z)\to \alg F_{\vv V}(Z)$ is equally general to $g:\free_{\vv V}(z) \to \alg P$, as testified by the homomorphisms $i, j$ which are such that: $\Alg(s)=i\circ g$ (by the definition of $s = i \circ g (z)$) and $g=j \circ i \circ g = j\circ \Alg(s)$. Thus $\Alg$ is surjective and the first claim is proved.

	As for (2), consider an algebraic e-generalization problem $h: \free_{\vv V}(z)\to \prod_{k = 1}^m \alg E_k$, where $\alg E_k$ is an $e_k$-subalgebra of $\free_{\vv V}(X_k)$. By direct application of the definitions, note that $\Alg\circ \Sym (h): \free_{\vv V}(z) \to \prod_{k = 1}^m \alg E(t_k)$ is the homomorphism defined by mapping $z$ to $(t_1, \ldots, t_m)$ where $t_k = e_k \circ p_k\circ h(z)$ for $k = 1, \ldots, m$. Notice that the embeddings $e_k: \alg E_k \to \alg F_{\vv V}(X_k)$ can be seen as isomorphisms onto their image; to avoid confusion, let us denote these isomorphisms as $l_k: \alg E_k \to \alg E(t_k)$, mapping
	$p_k \circ h(z)$ to $t_k$; let the maps $l_k^{-1}: \alg E(t_k) \to \alg E_k$ be their inverses, mapping $t_k$ back to $p_k \circ h(z)$, for $k = 1, \ldots, m$. The associated maps on the direct products:
	\begin{eqnarray*}
		l:=(l_1 \circ p_1, \ldots, l_m \circ p_m)\;&:&\; \prod_{k = 1}^m \alg E_k \to \prod_{k = 1}^m \alg E(t_k),\\
		l^{-1}:= (l_1^{-1}\circ p_1, \ldots, l_m^{-1}\circ p_m)\;&:&\; \prod_{k = 1}^m \alg E(t_k) \to \prod_{k = 1}^m \alg E_k
	\end{eqnarray*} 
	define inverse isomorphisms, and are such that $l \circ h = \Alg\circ \Sym (h)$ and $l^{-1} \circ \Alg\circ \Sym (h) = h$. 

Indeed, note that 
$$
l\circ h(z)=(t_1,\ldots, t_m)=\Alg\circ\Sym(h)(z)
$$
 and 
 $$
 l^{-1}\circ \Alg\circ \Sym(h)(z)=l^{-1}(t_1,\ldots, t_m)=h(z).
$$

It follows that any $g: \alg F_{\vv V}(z) \to \alg P$ that is a solution to $h$ is also a solution to $\Alg\circ \Sym (h)$. Indeed, if $f$ witnesses $g$, and hence it is such that $f \circ g = h$, then $l \circ f \circ g = l \circ h = \Alg\circ \Sym (h)$  and hence $l\circ f$ witnesses $g$ as solution of $\Alg\circ \Sym (h)$. Conversely, if $g': \alg F_{\vv V}\to \alg P'$ is solution to $\Alg\circ \Sym (h)$ witnessed by $f'$, then $l^{-1} \circ f'$ witnesses that $g'$ is solution to $h$ as well.

	Therefore, $(\mathscr{A}(h), \sqsubseteq)=(\mathscr{A}(\Alg\circ \Sym (h)), \sqsubseteq)$, the latter being isomorphic to $(\mathscr{S}(\Sym (h)), \preceq)$ as shown above.
\end{proof}
Therefore,  the poset of all solutions for a symbolic problem, and in particular the existence of a least general solution, can be studied algebraically. 
We can  then derive that the same holds also for the {\em e-generalization type} of a variety $\vv V$, which is the worst occurring type for its e-generalization problems. 
\begin{corollary}\label{cor:type}
	Given a variety $\vv V$, its symbolic and  algebraic e-generalization $\kappa$-types coincide, for any cardinal $\kappa \leq \omega$.
\end{corollary}
\begin{proof}
The above Theorem \ref{them:posetiso} demonstrates that given a symbolic e-generalization problem for $\vv V$, there is an algebraic one with the same type and vice versa, which implies the claim.
\end{proof}

	The reader familiar with the algebraic study of unification problems developed by Ghilardi in \cite{G97} will appreciate that our results have a similar flavor, but also some significant differences. Importantly, the study of unification problems and their type only involves homomorphisms between finitely presented and projective algebras in a variety; since these are all {\em categorical} notions, it directly follows that the unification type is preserved under categorical equivalence. With the approach outlined above, we do not get this same result, since we make a key use of 1-generated (free) algebras, and being 1-generated is not generally preserved by categorical equivalences.
	Nonetheless, the e-generalization type is preserved under a stronger notion of categorical equivalence, called {\em equivalence} in \cite{Mckenzie}, which is quite common among algebraic categories. Let us be more precise. For the unexplained notions of category theory we refer the reader to \cite{Maclane}.

	Given any class of algebras $\vv K$, one can consider the associated algebraic category, denoted by the same symbol, whose objects are the algebras in $\vv K$ and whose morphisms are the homomorphisms.  
	Two algebraic categories $\vv K$ and $\vv L$ are {\em equivalent} if they are categorically equivalent via a functor $\Gamma$ that preserves free algebras, in the sense that $\Gamma(\free_{\vv K}(X)) = \free_{\vv L}(X)$ for all sets $X$. 
If one considers in particular two varieties $\vv V$ and $\vv W$, these are equivalent if and only if their algebraic theories are isomorphic (from a category theoretical perspective), or if their clones are isomorphic (from a universal algebraic perspective), see \cite{Mckenzie} for a discussion.

All the notions involved in the definition of an algebraic e-generalization problem, its solutions, and the generality order, are preserved by equivalence; therefore we have the following result.
\begin{theorem}\label{thm:equivalence}
	The e-generalization $\kappa$-type of a variety is preserved by equivalence, for any cardinal $\kappa \leq \omega$.
\end{theorem}
\begin{remark}\label{remark:freeinsteadofprojective}
We conclude this section with the following observation. One could have restricted the algebraic solutions in Definition \ref{def:functunif} to be homomorphisms among free algebras, $g: \alg F_{\vv V}(z) \to \alg F_{\vv V}(Y)$, for some finite set $Y$. This choice would yield equivalent results with respect to the study of the generality poset of solutions. Indeed, for the non-trivial side, consider a solution $g: \alg F_{\vv V}(z) \to \alg P$ for $\alg P$ finitely generated and projective. Then $\alg P$ is an $(i,j)$-retract of some finitely generated free algebra $\alg F_{\vv V}(Y)$. We claim that the composition $i \circ g: \alg F_{\vv V}(z) \to \alg F_{\vv V}(Y)$ is equally general to $g$. Indeed, while by definition of the order it follows directly that $i \circ g \sqsubseteq g$, notice that $g = j \circ (i \circ g)$ and hence $g \sqsubseteq i \circ g$.

However, the presentation involving projective algebras has several advantages. First, it allows one to study generalization problems in varieties where one does not have a handy description of free algebras, but instead has a clear characterization of projective ones. For instance, in some locally finite varieties the finitely generated projective algebras are all the finite algebras, but there is no general description of free algebras; see the example of {\em hoops} in Section \ref{sec:godel}, or Section \ref{sec:kleene}, where we use a dualized representation of projective algebras to characterize the type of Kleene algebras.

Finally, another important observation is that projective algebras are preserved by categorical equivalences among algebraic categories, while free algebras are not. Therefore, we prefer to use the presentation of the solutions with projective algebras; the reader may keep in mind, however, that a presentation with free algebras alone would be equivalent.
\end{remark}

\section{A congruence-based approach to the poset of solutions}\label{section:congruence}
In this section we  exploit the insight given by the algebraic representation of  e-generalization problems and develop a study of the generality poset of solutions guided by the congruence lattice of the 1-generated free algebra in the considered variety. 

We start by characterizing the kernels of e-generalization problems and their solutions, and using them to study the poset of solutions. 
Then, we identify a class of varieties, later named {\em 1ESP varieties}, where  the study of the generality poset of solutions completely reduces to the study of projective congruences of the 1-generated free algebra. 
Finally, we develop an algebraic analysis of the generalization type, deriving some general results and a sufficient condition for a variety to have unitary type, i.e., finite direct products of 1-generated exact algebras being projective.

In particular, we will demonstrate that the following varieties have unitary e-generalization type: abelian groups, commutative semigroups and monoids, and all varieties whose 1-generated free algebra is trivial (e.g., lattices, semilattices, varieties without constants whose operations are idempotent).

\subsection{E-generalization problems via congruences}
Fix a variety $\vv V$ and an algebraic e-generalization problem $h: \free_{\vv V}(z) \to \prod_{k = 1}^m \alg E_k$; its kernel $\ker(h)$ is a congruence in the 1-generated free algebra $ \free_{\vv V}(z)$. The following lemma comes from a first easy observation, namely that the kernel of the problem gives an upper bound to the kernels of the solutions in the  congruence lattice of $ \free_{\vv V}(z)$.

\begin{lemma}\label{thm:trivialSol1}
Let  $h: \free_{\vv V}(z) \to \prod_{k = 1}^m \alg E_k$ be an algebraic e-generalization problem for a variety $\vv V$. If $g: \free_{\vv V}(z)  \to \alg P$ is a solution, then $\ker(g)\subseteq \ker(h)$.
\end{lemma}
\begin{proof}
The claim simply follows from the fact that if $g: \free_{\vv V}(z)  \to \alg P$ is a solution, there is a testifying homomorphism $f: \alg P \to \prod_{k = 1}^m \alg E_k$ such that $f \circ g = h$.
\end{proof}

We are now going to characterize which congruences of the 1-generated free algebra can appear as $\ker(h)$ for an e-generalization problem given by a homomorphism $h$, but first we need a technical lemma.

\begin{lemma}\label{lemma:charcong}
Let $\vv V$ be any variety. Let $h: \alg A \to \prod_{k = 1}^m \alg B_k$ be a homomorphism, and consider $h_k= p_k \circ h$ for $k = 1, \ldots, m$, where $p_k$ is the $k$-th projection. Then $\ker(h) = \bigcap_{k = 1}^m \ker(h_k)$.
\end{lemma}

\begin{proof}
	Since $h_k= p_k \circ h$, it directly follows that $\ker(h) \sse \ker(h_k)$ for all $k= 1,\ldots, m$, and then $\ker(h) \sse \bigcap_{k = 1}^m \ker(h_k)$. For the converse, let $a,a'\in A$ and assume that  $(a,a') \notin \ker(h)$, i.e., $h(a) \neq h(a')$. This implies that $h(a)$ and $h(a')$ differ in some factor, that is to say, $p_k\circ h(a)=h_k(a) \neq h_k(a')=p_k\circ h(a')$ for some $k\leq m$. Therefore $(a,a') \notin \bigcap_{k = 1}^m \ker(h_k)$, which shows that $\bigcap_{k = 1}^m  \ker(h_k) \sse \ker(h)$ and concludes the proof.
\end{proof}
We are now ready to characterize which congruences can appear as kernels of algebraic e-generalization problems; let us make this notion precise for future reference.
\begin{definition} Let $\vv V$ be a variety. 
	We call a congruence $\theta \in \Con{\free_{\vv V}(z)}$  an {\em E-congruence} if $\theta = \ker(h)$ for some e-generalization problem  $h: \free_{\vv V}(z) \to \prod_{k = 1}^m \alg E_k$.
\end{definition}

As a consequence of Lemma \ref{lemma:charcong} above, we obtain the following.

\begin{theorem}\label{propThetat}
For every $\theta\in\Con{\free_{\vv V}(z)}$, $\theta$ is an E-congruence if and only if there are exact congruences $\theta_1,\ldots,\theta_m\in \Con{\free_{\vv V}(z)}$ such that $\theta=\bigcap_{k=1}^m \theta_m$.
\end{theorem}
\begin{proof}
Let $\theta$ be an E-congruence and let $h:\free_{\vv V}(z)\to \prod_{k=1}^m \alg E_k$ be an algebraic e-generalization problem such that $\theta=\ker(h)$. By Lemma \ref{lemma:charcong}, $\ker(h) = \bigcap_{k = 1}^m \ker(h_k)$ where $h_k=p_k\circ h$ for all $k=1,\ldots, k$. Since $h_k: \free_{\vv V}(z)\to \alg E_k$ is onto, by the First Isomorphism Theorem we get that  $\free_{\vv V}(z)/ \ker(h_k)$ is isomorphic to the exact algebra $\alg E_k$, hence exact itself, and then $\ker(h_k)$ is exact by definition. Thus, the claim follows by taking $\theta_k=\ker(h_k)$ for all $k=1,\ldots,m$.

Conversely, if $\theta=\bigcap_{k=1}^m \theta_k$ for $\theta_k$ exact congruences,  by definition of exact congruence there are embeddings $e_k: \free_{\vv V}(z)/ \theta_k \to \alg F_{\vv V}(X)$ for some $X$ ($X$ can be taken to be the same for all $k$). Thus, letting $n_k: \free_{\vv V}(z)\to \free_{\vv V}(z)/ \theta_k$ be the natural epimorphism, one can define an e-generalization problem $h: \free_{\vv V}(z) \to \prod_{k = 1}^m \free_{\vv V}(z)/ \theta_k$ as $h(x) = (n_1(x), \ldots, n_m(x))$, since clearly each $p_k \circ h$ is onto $\free_{\vv V}(z)/ \theta_k$. Moreover: 
$$
\theta = \bigcap_{k=1}^m \theta_k =\bigcap_{k=1}^m \ker(n_k)=\bigcap_{k=1}^m \ker(p_k\circ h)=\ker(h),
$$ 
 using Lemma \ref{lemma:charcong} in the last identity. Therefore, $\theta$ is an E-congruence.
\end{proof}
In other words, a congruence $\theta$ of the 1-generated free algebra is the kernel of an algebraic e-generalization problem if and only if it is a finite intersection of exact congruences.

\subsection{Algebraic solutions via congruences}

Now  we turn to describing those congruences that are kernels of solutions of an algebraic problem.  
\begin{definition}
	Let  $h: \free_{\vv V}(z) \to \prod_{k = 1}^m \alg E_k$ be an algebraic e-generalization problem for a variety $\vv V$. We say that $\theta \in \Con{\free_{\vv V}(z)}$ is a {\em generalizing congruence for $h$} (or {\em G-congruence for $h$} for short) if $\theta = \ker(g)$ for some solution $g: \free_{\vv V}(z)  \to \alg P$. Moreover, we denote the set of G-congruences for $h$ by $\Gen(h)$.
\end{definition}
Note that $\mathscr{G}(h)$ is the image under the function $\ker$ of the set $\mathscr{A}(h)$, given by all generalizers for the e-generalization problem $h$: 
$$
\mathscr{G}(h)=\ker[\mathscr{A}(h)].
$$ 
Next theorem presents, respectively, necessary and sufficient conditions for a congruence to be a G-congruence of an e-generalization problem.  
\begin{theorem}\label{thm:generalizercong}
	Let  $h: \free_{\vv V}(z) \to \prod_{k = 1}^m \alg E_k$ be an algebraic e-generalization problem for a variety $\vv V$ and let  $\theta \in \Con{\free_{\vv V}(z)}$. The following hold:
\begin{enumerate}
\item If $\theta\in \Gen(h)$, then $\theta$ is exact and $\theta \sse \ker(h)$. 
\item If $\theta$ is projective and  $\theta \sse \ker(h)$, then $\theta\in \Gen(h)$.
\end{enumerate}
\end{theorem}
\begin{proof}
(1)	Suppose  that $\theta$ is a G-congruence for $h$. Thus, by definition, $\theta = \ker(g)$ for some solution $g: \free_{\vv V}(z)  \to \alg P$, and then by Lemma \ref{thm:trivialSol1}, $\theta = \ker(g) \sse \ker(h)$. Moreover,  $\free_{\vv V}(z)/\ker(g)$ embeds into $\alg P$ by direct application of  the First Isomorphism Theorem.  In turn,  $\alg P$  embeds into  some free algebra $\free_{\vv V}(Y)$. Therefore, $\ker(g)$ is exact.
\vspace{.1cm}
	
(2)	Suppose $\theta \in \Con{\free_{\vv V}(z)}$ is projective and $\theta \sse \ker(h)$. Then, $\free_{\vv V}(z)/\theta$ is projective and  by the Second Isomorphism Theorem (Theorem \ref{thm:secondiso}) 
 there is $f: \free_{\vv V}(z)/\theta \to \prod_{k = 1}^m \alg E_k $ such that $f\circ n_\theta=h$. Hence $f$ testifies that $n_\theta$ is a solution for $h$, and therefore $\theta = \ker(n_\theta)$ is a G-congruence for $h$.
\end{proof}

Of particular interest are then those varieties where the necessary and sufficient conditions in Theorem \ref{thm:generalizercong} collapse. This happens in particular whenever  every 1-generated exact algebra is projective; in this case we get the following characterization of G-congruences.
\begin{corollary}\label{cor:1EPchar}
	Let $\vv V$ be a variety such that every 1-generated exact algebra is projective. Then given any e-generalization problem $h$, $$\Gen(h) = \{\theta \in \Con{\free_{\vv V}(z)}: \theta \sse \ker(h), \theta \mbox{ projective in } \vv V\}.$$
\end{corollary}

Let us call the above property, namely that all 1-generated exact algebras are projective, {\em 1EP property}, and the varieties that satisfy it {\em 1EP varieties}. 1EP varieties seem to abound. 
They include all varieties where the 1-generated free algebra is trivial (e.g., lattices and semilattices), and moreover:
\begin{example}\label{exGruppi1}
Groups, abelian groups, monoids, semigroups, and their commutative subvarieties are easily seen to be 1EP varieties. 

Let us first consider the varieties of groups and abelian groups. The 1-generated free algebra is isomorphic to the additive group of the integers, $\alg Z=(\mathbb{Z}, +, -, 0)$ in both cases. Its quotients  are: $\alg Z$ itself; the trivial group $\alg 0$; for every $n$,  the finite cyclic groups $\alg Z_n$ of each order $n$. While clearly $\alg Z$ and $\alg 0$ are subalgebras of the 1-generated free algebra $\alg Z$, and hence they are exact, no finite  nontrivial group can be a subalgebra of a free group. Therefore, the exact quotients of $\alg Z$ are itself and $\alg 0$, which are both projective since they are free ($\alg 0$ is the 0-generated free (abelian) group). Thus both  groups and abelian groups are 1EP. 

Now, the 1-generated free algebra for (commutative) monoids is the additive monoid of the natural numbers $\alg N=(\mathbb{N}, +, 0)$, and the free (commutative) semigroup is the additive semigroup of positive natural numbers $\alg N^+=(\mathbb{N}-\{0\}, +)$.
Therefore, the situation is similar to the above: the 1-generated exact algebras in monoids and commutative monoids are the trivial algebra and the free algebra itself, which are both free hence projective; all other quotients are finite nontrivial algebras that cannot embed in a free monoid. For semigroups and their commutative subvariety, the trivial algebra is not exact, so the only exact (and projective) 1-generated algebra is the free algebra. Therefore, all these varieties are 1EP.
\hfill $\Box$ \end{example}
Moreover, many varieties coming from logic are 1EP; for instance, Heyting algebras (since all exact Heyting algebras are projective, \cite[Remark pp. 867-868]{G99}), and as we shall see in the last section, Boolean, G\"odel, and Kleene algebras. However, one can construct examples of varieties without the 1EP property; the following is due to Tommaso Moraschini  (private communication).

\begin{example} 
	Consider the monoid $\alg N_3 = ((0,1,2,3), \oplus, 0)$ where $\oplus$ is the addition between natural numbers truncated at $3$, and let $\vv{N}_3$ be the variety it generates. The 1-generated free algebra, $\free_{\vv{N}_3}(z)$ is isomorphic to $\alg N_3$ via the map sending $z$ to $1$. Consider the subalgebra $\alg S$ of $\alg N_3$ with domain $(0,2,3)$. This is a 1-generated exact algebra, since it embeds into a free algebra; moreover, it can be verified that $\alg S$ is not projective. Indeed, consider the surjective homomorphism $j: \alg N_3 \to \alg S$ defined by $j(1) = 2$. Then, if $\alg S$ were to be projective, it would have to exist an embedding $i: \alg S \to \alg N_3$ such that $j \circ i = {\rm id}_{\alg S}$. Since $j^{-1}(2) = \{1\}$, one would have to set $i(2) = 1$. But this assignment does not extend to a homomorphism since one would get  $$i(3) = i(2 \oplus 2) = i(2) \oplus i(2)  = 2 \neq 3 = i(2) \oplus i(2) \oplus i(2) =i(2 \oplus 2 \oplus 2) = i(3).$$
	Therefore  $\vv{N}_3$ is not 1EP. 
	It is interesting to further notice that, however, $\alg S$ is {\em weakly projective} in $\vv{N}_3$\footnote{We note that we were not able to find in the literature explicit examples of algebras that are weakly projective but not projective, hence this seemed worth observing. Weakly projective algebras are important in the study of primitive quasivarieties (see e.g., \cite[Prop. 5.1.24]{Gorbunov} or \cite[Theorem 4.11]{AgUg24}).}, i.e., whenever it is a homomorphic image of an algebra in the variety it is also (isomorphic to) a subalgebra. This is a consequence of the fact that every algebra in  $\vv{N}_3$ satisfies the identity $3x \approx 4x$ (where $nx$ is $x \oplus \ldots \oplus x$, $n$ times). Indeed, let $\alg A \in \vv{N}_3$ and $h$ a surjective homomorphism $h: \alg A \to \alg S$; consider $a \in A$ such that $h(a) = 2$. If $2a = 3a$ then the desired embedding $f: \alg S \to \alg A$ is defined by mapping $2$ to $a$, and otherwise, since necessarily $3a = 4a$, $f$ is defined by mapping $2$ to $2a$.  \hfill $\Box$
\end{example}
The above result indicates that in 1EP varieties one can fully characterize G-congruences. However, notice that this does not mean that we can characterize {\em solutions} via G-congruences. 
 Fix again an e-generalization problem $h$, and consider a homomorphism $g: \free_{\vv V}(z) \to \alg P$, with $\alg P$ projective; even if $\ker(g) \in \Gen(h)$, $g$ is not necessarily a solution. We will now identify a class of varieties where this always happens. 
 
 In order to build the necessary intuition, observe first than in general a solution $g:\free_{\vv V}(z)\to \alg P$ to an algebraic problem $h$ is not necessarily a surjective map, but by the First Isomorphism Theorem, it can always be decomposed as $g = e \circ n_{\ker(g)}$, 
where $n_{\ker(g)}$ is the natural epimorphism for $\ker(g)$ and $e: \free_{\vv V}(z)/\ker(g) \to \alg P$ is an embedding. Recall that, by definition, $g$ is a solution to $h$ if and only if there exists $f$ such that $f \circ g = h$ as in Figure \ref{fig1ESP}.
\begin{figure}[h!]
\begin{center}
$$
\xymatrix{
\free_{\vv V}(z)\ar[rr]^{h}\ar@/_{3pc}/[ddd]_{g}\ar[d]^{n_{\ker(g)}}& &\displaystyle\prod_{k = 1}^m \alg E_k \\
\free_{\vv V}(z)/\ker(g) \ar@<-0.6ex>[dd]_-e \ar[rru]_-{f'}&&  \\
&&\\
\alg P \ar[rruuu]_{f} \ar@<-0.6ex>[uu]_-j& &
 }
$$
\end{center}
\caption{Solutions in 1EP varieties.}\label{fig1ESP}
\end{figure}
Note that, if the variety is 1EP,  $n_{\ker(g)}$ is a solution whenever $\ker(g)\in \Gen(h)$. Indeed,  $\free_{\vv V}(z)/\ker(g)$ is projective, and by the Second Isomorphism Theorem, there is $f'$ such that $f'\circ n_{\ker(g)}=h$ which testifies that $n_{\ker(g)}$ is a solution. Now, if there  exists $j:\alg P\to \free_{\vv V}(z)/\ker(g)$ that makes $\free_{\vv V}(z)/\ker(g) $ an $(e,j)$-retract of $\alg P$, then the composition $f'\circ j$ testifies that $g$ is a solution as well.  Thus, we identify a class of varieties where we can always perform the last step of the above reasoning, i.e., $\free_{\vv V}(z)/\ker(g)$ being an $(e,j)$-retract of $\alg P$ for some $j$.

\begin{definition}
We say that an algebra $\alg S$ in a variety $\vv V$  is {\em strongly projective} in  $\vv V$ if it is projective in $\vv V$, and whenever there is an embedding $i: \alg S\to \alg P$ for some projective algebra $\alg P$, there is a homomorphism $j:\alg P \to \alg S$ such that $j\circ i=id_{\alg S}$. 

We say that a variety $\vv V$ is 1ESP if all 1-generated exact algebras in $\vv V$ are strongly projective in $\vv V$.
\end{definition} 
 Observe that that 1ESP implies 1EP, while the two notions are distinct as the following example shows.

\begin{example}\label{exGruppi2}
Recall from Example \ref{exGruppi1} that (abelian) groups are 1EP. From the same example, the 1-generated exact algebras are $\alg Z$ and $\alg 0$. While $\alg 0$ can be easily seen to be strongly projective, $\alg Z$ is not. As a matter of fact, consider any embedding $i$ of $\alg Z$ into itself defined by mapping $1$ to some positive integer $n>1$. Notice that there cannot exist a homomorphism $j: \alg Z \to \alg Z$ such that $j\circ i=id_{\alg Z}$.  Indeed, take for instance $n=2$. Then, if such a $j$ existed, one would have $1=j\circ i(1)=j(2)=j(1)+j(1)$ that is a clear contradiction.  Therefore the variety of (abelian) groups is 1EP, but not  1ESP.
\hfill $\Box$ \end{example}
In 1ESP varieties we can characterize solutions by means of their kernels.
\begin{theorem}\label{thm:1ESPsol}
	Let $\vv V$ be a 1ESP variety. Consider an e-generalization problem $h$. Thus, $g: \free_{\vv V}(z) \to \alg P$, with $\alg P$ projective in $\vv V$, is a solution to $h$ if and only if $\ker(g) \in \Gen(h)$. 
\end{theorem}
\begin{proof}
	If $g$ is a solution then $\ker(g)$ is in $\Gen(h)$ by definition. Conversely, assume $\ker(g) \in \Gen(h)$, and let us decompose $g:= e \circ n_{\ker(g)}$ as in the above Figure \ref{fig1ESP}. Since 1ESP implies 1EP, Corollary \ref{cor:1EPchar} yields that $\ker(g)$ is projective and below $\ker(h)$. By the Second Isomorphism Theorem, we get that there exists a homomorphism $f'$ such that $f' \circ n_{\ker(g)} = h$. Since $\free_{\vv V}(z)/\ker(g)$ is projective (and exact), and embeds via $e$ into the projective algebra $\alg P$, by 1ESP there exists $j$ such that $j \circ e$ is the identity map. Hence, $f:= f' \circ j$ testifies that $g$ is a solution, indeed: $$f \circ g = f' \circ j \circ e \circ n_{\ker(g)} = f' \circ n_{\ker(g)} = h.$$
	This concludes the proof.
\end{proof}
As a consequence we get the following invariance result for e-generalization problems in 1ESP varieties.
\begin{corollary}\label{cor:invariant1}
	Let $\vv V$ be 1ESP. Two algebraic e-generalization problems $h: \free_{\vv V}(z) \to \prod_{k = 1}^m \alg E_k$ and $h': \free_{\vv V}(z) \to \prod_{k = 1}^r \alg E_k$ have the same solutions if and only if $\ker(h)=\ker(h')$.
\end{corollary}
Let us highlight some examples of 1ESP varieties.
\begin{example}\label{ex:trivialfree}
Every variety $\vv V$ whose 1-generated free algebra is trivial has the 1ESP property. To see this, notice that $\free_{\vv V}(z)$, whose domain is $\{z\}$, is the only 1-generated exact algebra in $\vv V$ and it is strongly projective. Indeed, let $\alg A$ be any algebra in $\vv V$ and let  $i: \free_{\vv V}(z)\to \alg A$ be an embedding. The trivial homomorphism $j: \alg A\to  \free_{\vv V}(z)$ is such that $j\circ i=id_{\free_{\vv V}(z)}$, and then $\vv V$ is 1ESP. 

This reasoning covers, for instance, the cases of lattices, semilattices, and more generally all varieties without constants and whose operations are idempotent (in the sense that, given an $n$-ary operation $f(x_1, \ldots, x_n)$, $f(x, \ldots, x) = x$).
\hfill$\Box$
\end{example}
\begin{example}\label{prop:injproj}
Varieties where 1-generated exact algebras are {\em injective} are 1ESP.
Recall that an algebra $\alg A$ in a class of algebras $\vv K$ is said to be {\em injective} in $\vv K$ if, for any $\alg B, \alg C \in \vv K$, homomorphism $f : \alg B \to \alg A$ and  embedding $g : \alg B \to \alg C$, there exists a homomorphism $h : \alg C \to \alg A$ such that $f = h\circ g$.  

Then, in a variety $\vv V$, if an algebra $\alg A$ is exact and injective in $\vv V$, we get that $\alg A$ is strongly projective in $\vv V$. Indeed, since $\alg A$ is exact, it embeds into some (finitely generated) free algebra $\free_{\vv V}(X)$. Call the embedding $g$, and consider as $f$ the identity map from $\alg A$ to $\alg A$. Then since $\alg A$ is injective there exists $h: \free_{\vv V}(X) \to \alg A$ such that $id_{\alg A} = h \circ g$, which implies that $\alg A$ is a retract of $\free_{\vv V}(X)$ and it is therefore projective. Similarly, if there is an embedding $i: \alg A\to \alg P$ for some projective algebra $\alg P$, there is $j:\alg P \to \alg A$ such that $j\circ i=id_{\alg A}$ and therefore $\alg A$ is strongly projective.

This applies for instance to Boolean algebras, where all finite(ly generated) algebras are injective \cite{BalbesHorn}.
\end{example}
\subsection{A universal-algebraic methodology for the e-generalization type}
Let  $h: \free_{\vv V}(z) \to \prod_{k = 1}^m \alg E_k$ be an algebraic e-generalization problem for a variety $\vv V$. Given the results in the previous subsections, we shall develop a methodology to study the poset of solutions $(\mathscr{A}(h),\sqsubseteq)$ via the set $\Gen(h)$. 

To this end, we first observe that one can partition the algebraic solutions by means of their kernels. For every G-congruence $\theta\in \Gen(h)$, consider the following set:
$$
\Pi_\theta(h)=\{g\in \mathscr{A}(h)\mid  \ker(g)=\theta\}.
$$
The set $\{\Pi_\theta(h): \theta \in \Gen(h)\}$ is a partition of $\mathscr{A}(h)$.
Moreover, the inclusion order among the G-congruences of an algebraic e-generalization problem $h$  gives some insight on the generality order among solutions belonging to different classes of the partition, as the following result shows.
\begin{proposition}\label{prop:poset11}
Let  $h: \free_{\vv V}(z) \to \prod_{k = 1}^m \alg E_k$ be an algebraic e-generalization problem for a variety $\vv V$. Given two solutions $g: \free_{\vv V}(z) \to \alg P$ and $g': \free_{\vv V}(z)\to \alg P'$ in $(\mathscr{A}(h),\sqsubseteq)$, the following conditions hold:
\begin{enumerate}
\item  if $g' \sqsubseteq g$, then $\ker(g)\subseteq\ker(g')$;
\item if  $\ker(g)\subseteq \ker(g')$, then $g' \sqsubseteq n_{\ker(g)}$. 
\end{enumerate}
\end{proposition}
\begin{proof}
(1) By definition of the order in the poset of algebraic solutions, if $g' \sqsubseteq g$ there exists $f:\alg P \to \alg P'$ such that $f \circ g = g'$. Hence, $\ker(g)\subseteq \ker(g')$.
\vspace{.2cm}

(2) Consider the natural epimorphism $n_{\ker(g)}:\free_{\vv V}(z)\to \free_{\vv V}(z)/\ker(g)$; since $\ker(g)\subseteq\ker(g')$, the Second Isomorphism Theorem  yields directly that there exists a homomorphism $f:  \free_{\vv V}(z)/\ker(g)  \to \alg P'$ such that $f \circ n_{\ker(g)} = g'$. Thus, $g' \sqsubseteq n_{\ker(g)}$ by definition of the generality order. 
\end{proof}

Note that in point (2) of Proposition \ref{prop:poset11}, $n_{\ker(g)}$ is not necessarily a solution itself. This indeed happens exactly when $\free_{\vv V}(z)/\ker(g)$ is a projective algebra. In general the following holds.
\begin{proposition}\label{theorem:trivialSol2}
Let  $h: \free_{\vv V}(z) \to \prod_{k = 1}^m \alg E_k$ be an algebraic e-generalization problem for a variety $\vv V$.  Given any $\theta \in \Gen(h)$, $n_\theta: \free_{\vv V}(z) \to \free_{\vv V}(z)/\theta$ is a solution if and only if $\theta$ is projective. Moreover, if that is the case, $n_\theta$ is the top element of $\Pi_\theta$ by the generality order.
\end{proposition}
\begin{proof}
Consider $\theta \in \Gen(h)$. Note that projectivity is a necessary condition for $n_\theta: \free_{\vv V}(z) \to \free_{\vv V}(z)/\theta$ to be a solution by definition of an algebraic solution. 

Assume now $\theta$ projective. Since $\theta \in \Gen(h)$, by definition there is some solution $g: \free_{\vv V}(z) \to \alg P$ such that $\theta = \ker(g)$, and a testifying homomorphism $f: \alg P \to \prod_{k = 1}^m \alg E_k$ such that $f \circ g = h$. By the Second Isomorphism Theorem applied to $n_\theta$ and $g$, the fact that $\theta = \ker(g)$ yields a homomorphism $f':  \free_{\vv V}(z)/\theta \to \alg P$ such that $f' \circ n_\theta = g$. Thus, the composition $f'' = f \circ f'$ testifies that $n_\theta$ is a solution to $h$. Indeed:
$$f'' \circ n_\theta= f \circ f' \circ n_\theta = f \circ g = h.$$
Therefore, if $\theta$ is projective and in $\Gen(h)$, $n_\theta$ is a solution.
The reasoning above also proves that $n_\theta$ is the top element of $\Pi_\theta$ by the generality order; indeed,  any $g: \free_{\vv V}(z) \to \alg P$ in $\Pi_\theta(h)$ is such that $\theta = \ker(g)$ and therefore, as above, there exists $f':  \free_{\vv V}(z)/\theta \to \alg P$ such that $f' \circ n_\theta = g$, and then $g \sqsubseteq n_\theta$.
\end{proof}
The above results allow one to have a ``rough'' idea of the generality poset of a problem. 
In the rest of the section we use this intuition towards the study of the e-generalization type.
In the general case, we can derive that the cardinality of a set of complete maximal congruences in $\Gen(h)$ gives a best-case scenario for the e-generalization type of $h$. To be precise, let us define the {\em type of $\Gen(h)$} to be:
\begin{itemize}
	\item unitary: if there is a maximal complete set of congruences ({\em mcsc} for short) in $(\Gen(h), \sse)$ of cardinality 1;
	\item finitary:  if there is an mcsc in $(\Gen(h), \sse)$ of finite (greater than $1$) cardinality;
	\item infinitary: if there is an mcsc in $(\Gen(h), \sse)$  of infinite (not finite) cardinality;
	\item nullary: if there is no mcsc in $(\Gen(h), \sse)$.
\end{itemize}
\begin{theorem}\label{thmType}
	Let  $h: \free_{\vv V}(z) \to \prod_{k = 1}^m \alg E_k$ be an algebraic e-generalization problem for a variety $\vv V$. Then:
	\begin{enumerate}
		\item If the type of $\Gen(h)$ is nullary or infinitary, then the e-generalization type of $h$ is at best infinitary.
		\item If the type of $\Gen(h)$ is finitary, the e-generalization type of $h$ is at best finitary.
	\end{enumerate}
\end{theorem}
\begin{proof}
Assume first that the type of $\Gen(h)$ is nullary or infinitary; then we prove that there cannot be a finite minimal complete set of solutions to $h$. Assume by way of contradiction that there is, and call it $\mathscr{M} = \{g_1, \ldots, g_n\}$, where each $g_i: \free_{\vv V}(z) \to \alg P_i$ for some projective algebra $\alg P_i \in \vv V$.
Consider now the set of congruences $\{\theta_i = \ker(g_{i}): i = 1, \ldots, n\} \sse \Gen(h)$; since the type of $\Gen(h)$ is nullary or infinitary, there must be an exact congruence $\theta \in \Gen(h)$ such that $\theta \not\sse \theta_i$ for all $i = 1, \ldots, n$. For otherwise one would have that the type of $\Gen(h)$ is finitary.

Now, since $\theta \in  \Gen(h)$, by definition there exists a solution $g: \free_{\vv V}(z) \to \alg P$ such that $\theta = \ker(g)$.
 Then $g_i \not\sqsubseteq g$ for all $i = 1, \ldots, n$. Indeed, if $g_i \sqsubseteq g$ for some $i \in \{1, \ldots, n\}$, by Proposition \ref{prop:poset11} we would get that $\theta \sse \theta_i$, a contradiction. Thus there is no element in $\mathscr{M}$ that is below $g$, which contradicts the fact that $\mathscr{M}$ is a complete set of minimal solutions. It follows that the type of $h$ is at best infinitary, and (1) holds.

For (2), suppose now that the type of $\Gen(h)$ is finitary. Then the same argument as above, setting $n=1$ and  $\mathscr{M}=\{g_1\}$, shows that the e-generalization type of $h$ cannot be unitary.
\end{proof}
Note that if the type of $\Gen(h)$ is unitary, we can draw no conclusion on the type of $h$ because we cannot make any assumptions on the generality poset restricted to $\Pi_\theta(h)$, where $\{\theta\}$ is a complete maximal set in $\Gen(h)$. Moreover, if the type of $\Gen(h)$ is nullary, we cannot directly conclude that the type of $h$ is also nullary. Indeed, for instance, it could be the case that in $\Gen(h)$ there is a countable cofinal chain of congruences $\theta_1 \sse \theta_2 \sse \ldots$ such that the union of the minimal elements in all the $\Pi_{\theta_i}(h)$ are all incomparable, and therefore constitute a (infinite) complete minimal set of solutions.

In 1ESP varieties the situation is considerably clearer, and in fact the study of the generality poset is fully reduced to the poset of projective congruences of the 1-generated free algebra.
\begin{theorem}\label{thm1ESP}
	Let $\vv V$ be a 1ESP variety and consider an algebraic e-generalization problem $h$. Its poset of solutions $\mathscr{A}(h)$ is dually isomorphic to the poset of congruences in $\Gen(h)$. 
	\end{theorem}
\begin{proof}
	We claim that the map $\ker$ establishes a dual isomorphism from $\mathscr{A}(h)$ to $\Gen(h)$. 
	First, note that the map is well defined since if two solutions $g, g' \in \mathscr{A}(h)$ are equally general, they have the same kernel. Indeed, $g \sqsubseteq g'$ implies $g = f \circ g'$ for some $f$, and therefore $\ker(g') \sse \ker(g)$; similarly $g' \sqsubseteq g$ implies $\ker(g) \sse \ker(g')$. Moreover, by Theorem \ref{thm:1ESPsol}, $g$ is a solution if and only if $\ker(g) \in \Gen(h)$, which implies that $\ker$ is a surjective function from $\mathscr{A}(h)$ to $\Gen(h)$. Let us demonstrate that the map is injective. Suppose that $g, g' \in \mathscr{A}(h)$ are such that $\ker(g) = \ker(g') = \theta$, we prove that $g$ and $g'$ are equally general. Let us write $g: = e \circ n_\theta, g':= e' \circ n_\theta$ for some embeddings $e, e'$, which exist by Theorem \ref{thm:secondiso}(1). By 1ESP, there exist $j, j'$ such that $j \circ e = id_{\free_{\vv V}(z)/\theta} = j' \circ e'$. Hence:
	$$e \circ j' \circ g' = e \circ j' \circ e' \circ n_\theta = e \circ n_\theta = g,$$
which verifies that $g \sqsubseteq g'$, and similarly $g' \sqsubseteq g$ since:
$$e' \circ j \circ g = e' \circ j \circ e \circ n_\theta = e' \circ n_\theta = g'.$$
We have proved that $\ker$ is a bijection from $\mathscr{A}(h)$ to $\Gen(h)$; it is a dual poset isomorphism as a consequence of Proposition \ref{prop:poset11}.
\end{proof}

\begin{corollary}\label{cor:anti1esp}
	Let $\vv V$ be 1ESP, and consider an e-generalization problem $h$. Then the e-generalization type of $h$ is exactly the type of $\Gen(h)$.
\end{corollary}	

\begin{remark}
	An interesting consequence of Theorem \ref{thm1ESP} above is that in 1ESP varieties each element of the generality poset has a representative given by one of the natural epimorphisms $n_\theta$, which has as codomain a 1-generated projective algebra. This means that, given a symbolic problem $\bt$, and a solution for it given by a term $s$, there is a 1-variable term $s'$ that is equally general to $s$. Essentially, in the poset of symbolic solutions one can restrict to consider 1-variable terms.
\end{remark}

\begin{corollary}\label{ex:typelattices}
	Every variety whose 1-generated free algebra is trivial has unitary e-generalization type. In particular: lattices, semilattices, all varieties without constants and whose operations are idempotent.
\end{corollary}
\begin{proof}
In Example \ref{ex:trivialfree} we have seen that every variety $\vv V$ whose 1-generated free algebra is trivial has the 1ESP property. This includes lattices, semilattices, and more generally all varieties without constants and whose operations are idempotent.
Since the 1-generated free algebra is trivial, the type of $\Gen(h)$ is clearly unitary for all problems $h$; it follows from Corollary \ref{cor:anti1esp} that every such variety has unitary $\kappa$-type for every cardinal $\kappa \leq \omega$.
\end{proof}
\begin{remark}
	It is proved in \cite{CernaKutsia} that idempotent anti-unification is infinitary, which seems to contradict the result outlined in the above corollary. However, the result in \cite{CernaKutsia} makes key use of two constants, assumed to be in the language, in order to show that the type is (at least) infinitary. If we add constants, the 1-generated free algebra is not trivial, and therefore there is no contradiction. We also note that our result is incomparable with the one in \cite{CernaKutsia}, in the sense that while ours is more general with respect to the fact that we do not restrict the equational theory to idempotency (our result holds for any equational theory that includes idempotency of all symbols), we do need to restrict to not having constants.
\end{remark}

Let us conclude this section by identifying a sufficient condition for deriving the unitary type of a problem.
\begin{lemma}\label{ThmTypeKer}
Let $h:\free_{\vv V}(z)\to\prod_{k=1}^m \alg E_k$ be an algebraic e-generalization problem. If $\prod_{k=1}^m \alg E_k$ is finitely generated and projective then the type of $h$ is unitary and $h$ is the minimum of $(\mathscr{A}(h),\sqsubseteq)$.
\end{lemma}
\begin{proof}
Note that if $\prod_{k=1}^m \alg E_k$ is finitely generated and projective then $h$ is a solution to itself, testified by the identity map. 
Moreover, let $g: \free_{\vv V}(z)\to \alg P$ be any other solution; then there exists $f:\alg P\to \prod_{k=1}^m \alg E_k$ such that $f\circ g=h$, and hence $h\sqsubseteq g$, whence $h$ is the minimum in $\mathscr{A}(h)$ with respect to the generality order.
\end{proof}
The previous lemma entails a sufficient condition for varieties to have unitary e-generalization type.
\begin{theorem}\label{corFinal}
If $\vv V$ is a variety such that finite direct products of 1-generated exact algebras are finitely generated and projective, then $\vv V$ has unitary type. 
\end{theorem}
\begin{proof}
If $h:\free_{\vv V}(z)\to\prod_{k=1}^m \alg E_k$ is any e-generalization problem, under our hypothesis, $\prod_{k=1}^m \alg E_k$ is finitely generated and projective and hence its type is unitary by Lemma \ref{ThmTypeKer}. Thus, the $\kappa$-type of $\vv V$, for all cardinals $\kappa \leq \omega$, is unitary as well.
\end{proof}

The following result closes this section, using the previous theorem to show the unitary type of the varieties of structures considered in the examples throughout this section\footnote{Baader in \cite{Baader} studies the generalization type of what are called {\em commutative theories}, which encompass commutative monoids and abelian groups, showing indeed unitary type. However, the generalization order considered in that paper is seemingly different from the one we consider here, as it involves ordering the testifying substitutions and not just the terms.}.
\begin{corollary}\label{ex:typegroups}
	Abelian groups, commutative monoids, and commutative semigroups have unitary e-generalization type.
\end{corollary}
\begin{proof}
	In Example \ref{exGruppi1} we have seen that in the varieties of interest the only 1-generated exact algebras are the free 1-generated algebra itself, $\free_{\vv V}(z)$ and, in all cases except commutative semigroups, the trivial algebra (which in such cases is projective). Thus, all nontrivial finite direct products of 1-generated exact algebras are finite powers of $\free_{\vv V}(z)$.

In all three cases, the $n$-th power $\free_{\vv V}(z)^n$ is isomorphic to the $n$-generated free algebra. To see this, notice that for all three varieties under consideration the $n$-generated free algebra is the algebra of polynomials (over the appropriate language) written over $n$-variables. Specifically, for abelian groups, commutative monoids, and commutative semigroups one gets $\mathbb{Z}[x_1,\ldots, x_n]$, $\mathbb{N}[x_1,\ldots, x_n]$, and $\mathbb{N}^+[x_1,\ldots, x_n]$ respectively. It is then easy to realize that these are isomorphic to, respectively, $\alg Z^n$, $\alg N^n$, and $(\alg N^{+})^n$. Therefore, finite direct products of 1-generated exact algebras are (free hence) projective, and we can apply Theorem \ref{corFinal} and conclude that all the three varieties above have unitary e-generalization type.
\end{proof}

Note that if we drop commutativity and consider groups, monoids, and semigroups, finite direct products of 1-generated exact algebras are not projective and hence one cannot apply Theorem \ref{corFinal}. Indeed,  the 1-generated exact algebras are the same as their commutative versions (see Example \ref{exGruppi1}) and hence so are their direct products, which are then commutative. However, these direct products  cannot embed into any free algebra since  no nontrivial subalgebra of free groups/monoids/semigroups is commutative. Hence finite direct products of 1-generated exact algebras are neither exact nor projective.

\section{Applications to (algebraizable) logics}\label{section:logic}
In this last section we will apply the results obtained so far to logic. 
Consider a logic ${\cc L}$ over a language $\rho$, identified by some substitution-invariant consequence relation on $\rho$-terms $\vdash\, \sse \cc P(\alg T_{\rho}(\omega)) \times \alg T_{\rho}(\omega)$ (for details on the definition of a logic in abstract algebraic logic see \cite{Font}). There is a natural equational theory that is usually associated to a logic, namely, {\em logical equivalence}. Given two formulas $\varphi$ and $\psi$ of the logic, that is to say, two terms in $\alg T_{\rho}(\omega)$, $\varphi$ is {\em logically equivalent} to $\psi$ if and only if $\varphi$ and $\psi$ have the same interpretation in every model of the logic.
 
 We consider an {\em e-generalization problem for a logic ${\cc L}$} to be an e-generalization problem with respect to the equational theory of logically equivalent formulas of $\cc L$.  
In the context of algebraizable logics (see \cite{BP89,Font} for details), the equational theory given by logical equivalence is exactly the one of the {\em equivalent algebraic semantics} of the logic $\cc L$, $\vv K_{\cc L}$. The associated variety is the smallest one containing $\vv K_{\cc L}$, or in other words, the variety generated by $\vv K_{\cc L}$\footnote{$\vv K_{\cc L}$ is in general a {\em generalized quasivariety}, and a {\em quasivariety} if the logic $\cc L$ is finitary (see \cite{Font} for details).}. When $\vv K_{\cc L}$ is a variety, we say that the logic is {\em strongly algebraizable} and the situation is simplified.

Hence, to study e-generalization problems for algebraizable logics we study the e-generalization type of their corresponding variety of algebras.

Given this general perspective, we now investigate some of the most well-known (algebraizable) logics and obtain examples of unitary type for the varieties of Boolean, G\"odel, and Kleene algebras, and nullary type for MV-algebras. 
\subsection{Classical logic}
The first case we take into consideration is that of classical logic and its equivalent algebraic semantics, the variety $\vv{BA}$ of Boolean algebras. 

By the reasoning in Example \ref{prop:injproj}, we get that Boolean algebras are 1ESP.
By Corollary \ref{cor:anti1esp}, the type of a problem $h$ for $\vv{BA}$ is exactly the type of $\Gen(h)$, which by Corollary \ref{cor:1EPchar} is given by the projective congruences below the E-congruence $\ker(h)$. 
Now, by Theorem \ref{propThetat}, the possible E-congruences are exactly the finite intersection of exact (hence projective) congruences of $\Con{\free_{\vv{BA}}(z)}$. In Figure \ref{figBoole} the reader can see both $\free_{\vv{BA}}(z)$ and its congruence lattice;  the projective congruences are all except the total one, $\nabla$. 

\begin{figure}
\begin{center}
\begin{tikzpicture}
  \node [label=below:{${0}$}, label=below:{} ] (n1)  {} ;
  \node [above left of=n1,label=left:{${z}$}, label=below:{ }] (n2)  {} ;
  \node [above right of=n1,label=right:{${\neg z}$},label=below:{}] (n3) {} ;
  \node [above left of=n3,label=above:{${1}$},label=below:{}] (n4) {} ;

  \draw  (n1) -- (n2);
  \draw (n1) -- (n3);
  \draw  (n3) -- (n4);
  \draw  (n2) -- (n4);
  \draw [fill] (n1) circle [radius=.5mm];
  \draw [fill] (n2) circle [radius=.5mm];
  \draw [fill] (n3) circle [radius=.5mm];
  \draw [fill] (n4) circle [radius=.5mm];
\end{tikzpicture}
\hspace{.5cm}
\begin{tikzpicture}

  \node [label=below:{${\Delta}$}, label=below:{} ] (n1)  {} ;
  \node [above left of=n1,label=left:{$\langle z, 1 \rangle$}, label=below:{ }] (n2)  {} ;
  \node [above right of=n1,label=right:{$\langle \neg z, 1 \rangle$},label=below:{}] (n3) {} ;
  \node [above left of=n3,label=above:{${\nabla}$},label=below:{}] (n4) {} ;

  \draw  (n1) -- (n2);
  \draw (n1) -- (n3);
   \draw  (n3) -- (n4);
  \draw  (n2) -- (n4);
  \draw [fill] (n1) circle [radius=.5mm];
  \draw [fill] (n2) circle [radius=.5mm];
  \draw [fill] (n3) circle [radius=.5mm];
  \draw [fill] (n4) circle [radius=.5mm];
\end{tikzpicture}
\caption{$\free_{\vv{BA}}(z)$ and its congruence lattice, where $\langle z, 1 \rangle$ and $\langle\neg z, 1 \rangle$ represent, respectively, the congruences generated by $(z,1)$ and $(\neg z,1)$.}\label{figBoole}
\end{center}
\end{figure}
It is clear that the E-congruences are exactly all and only the projective ones; moreover, the set of projective congruences below each one of them clearly has a maximum (themselves). Hence, applying Corollary \ref{cor:anti1esp} we derive, in analogy with the syntactic case \cite{Plotkin}:
\begin{theorem}
The e-generalization type of classical logic and Boolean algebras is unitary. 
\end{theorem}
Observe that one could have obtained the above result also utilizing Theorem \ref{corFinal}, since all finitely generated Boolean algebras, and hence in particular finite direct products of 1-generated exact Boolean algebras, are projective.

\subsection{G\"odel-Dummett logic and locally tabular many-valued logics}\label{sec:godel}
G\"odel-Dummett logic is one of the most well-studied intermediate logics, lying in the interval between intuitionistic and classical logic \cite{Du59}. Its language is given by the connectives $(\land, \lor, \to , 0, 1)$, and from the algebraic point of view its equivalent algebraic semantics, the variety $\vv{GA}$ of G\"odel algebras, is the subvariety of Heyting algebras generated by chains, i.e., totally ordered algebras, see \cite{Horn69}.

Since all non-trivial finite G\"odel algebras are projective (see \cite[Proposition 2.5]{DM06} for a proof based on a Priestley-like duality, and  \cite[Theorem 5.17]{AgUg1} for an algebraic proof), $\vv{GA}$ has the 1EP property. However, it is not hard to see that the 1ESP fails, as the following example shows.

\begin{example}\label{ex:GnotS}
Let us consider the 3-element chain ${\alg G}_{3}$ and the 4-element chain ${\alg G}_{4}$ with an embedding $i:{\alg G}_{3}\to{\alg G}_{4}$ as on the left-hand-side of the following figure
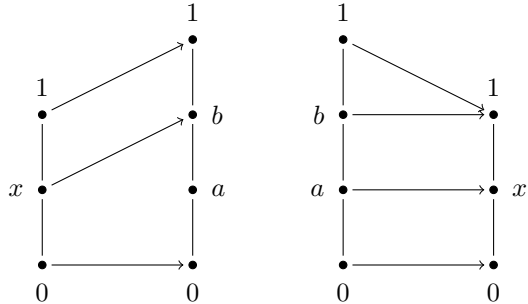
\begin{center}
\begin{tikzpicture}
  \node [label=below:{$0$}] (a1) at (-2,0) {};
\node [label=right:{$a$}] (a2) at (-2,1) {};
\node [label=right:{$b$}] (a3) at (-2,2) {};
\node [label=above:{$1$}] (a4) at (-2,3) {};
    \draw  (a1) -- (a2);
      \draw  (a2) -- (a3);
        \draw  (a3) -- (a4);
  \draw [fill] (a1) circle [radius=.5mm];
  \draw [fill] (a2) circle [radius=.5mm];
    \draw [fill] (a3) circle [radius=.5mm];
      \draw [fill] (a4) circle [radius=.5mm];
     \node [label=below:{$0$}] (b1) at (-4,0) {};
\node [label=left:{$x$}] (b2) at (-4,1) {};
\node [label=above:{$1$}] (b3) at (-4,2) {};
    \draw  (b1) -- (b2);
      \draw  (b2) -- (b3);
  \draw [fill] (b1) circle [radius=.5mm];
  \draw [fill] (b2) circle [radius=.5mm];
    \draw [fill] (b3) circle [radius=.5mm];

 \draw[->] (b1)--(a1);
  \draw[->] (b2)--(a3);
   \draw[->] (b3)--(a4);

\node [label=below:{$0$}] (m1) at (0,0) {};
\node [label=left:{$a$}] (m2) at (0,1) {};
\node [label=left:{$b$}] (m3) at (0,2) {};
\node [label=above:{$1$}] (m4) at (0,3) {};
    \draw  (m1) -- (m2);
      \draw  (m2) -- (m3);
        \draw  (m3) -- (m4);
  \draw [fill] (m1) circle [radius=.5mm];
  \draw [fill] (m2) circle [radius=.5mm];
    \draw [fill] (m3) circle [radius=.5mm];
      \draw [fill] (m4) circle [radius=.5mm];
     \node [label=below:{$0$}] (n1) at (2,0) {};
\node [label=right:{$x$}] (n2) at (2,1) {};
\node [label=above:{$1$}] (n3) at (2,2) {};
    \draw  (n1) -- (n2);
      \draw  (n2) -- (n3);
  \draw [fill] (n1) circle [radius=.5mm];
  \draw [fill] (n2) circle [radius=.5mm];
    \draw [fill] (n3) circle [radius=.5mm];
 \draw[->] (m1)--(n1);
  \draw[->] (m2)--(n2);
    \draw[->] (m3)--(n3);
      \draw[->] (m4)--(n3);
        \end{tikzpicture}
\end{center}
It can be easily checked that every homomorphic image of ${\alg G}_{4}$, that is not itself, collapses $b$ and $1$ as hinted at on the right-hand-side of the above figure. 
Hence, there is no homomorphism $j: {\alg G}_{4}\to {\alg G}_{3}$ such that $j\circ i$ is the identity, which entails that ${\alg G}_{3}$ is a 1-generated exact algebra that is not strongly projective. Therefore, $\vv{GA}$ does not have the 1ESP property.
\hfill $\Box$ \end{example}
Nonetheless, we can use Theorem \ref{corFinal} to prove that G\"odel algebras have unitary e-generalization type.
Indeed, all non-trivial finite G\"odel algebras are projective, and since G\"odel algebras are a locally finite variety, finite direct products of 1-generated exact algebras are finite. 
We can then conclude the following.
\begin{theorem}
The e-generalization type of G\"odel-Dummett logic and G\"odel algebras is unitary. 
\end{theorem}

More generally, one can prove uniformly the unitarity of the e-generalization type for all locally finite varieties of ($0$-bounded) {\em hoops}, of which G\"odel algebras and Boolean algebras are particular examples. 
Hoops are a variety of residuated monoids which, particularly in their $0$-bounded version, constitute the equivalent algebraic semantics of some of the most relevant many-valued logics. For instance, all logics in H\'ajek's framework, among which: \L ukasiewicz logic, G\"odel Dummett logic, Basic Logic (see \cite{Hajek,AFM2007}). 
While we preferred to isolate the above cases as a mean to better illustrate our methodology, and because they have a particular interest on their own, let us now present a coincise uniform argument for varieties of hoops.

The key is given by results in \cite{AgUg1}, specifically Theorem 5.3 and Corollary 5.15, which we summarize in the following statement.
\begin{theorem}[{\cite{AgUg1}}]\label{thm:aglianougolini-hoops}
	In every locally finite variety of ($0$-bounded) hoops, finitely generated projective algebras coincide with finite algebras (that have a homomorphism to the 2-element Boolean algebra). 
\end{theorem}
This allows one to apply Theorem \ref{corFinal}.
\begin{theorem}
	All locally finite varieties of hoops and $0$-bounded hoops have unitary e-generalization type.
\end{theorem}
\begin{proof}
	For locally finite variety of hoops, notice that finite direct products of $1$-generated algebras are finite in a locally finite variety, and hence projective by Theorem \ref{thm:aglianougolini-hoops}. Therefore, one can apply Theorem \ref{corFinal} to obtain the unitarity of the type. 
	
	For the $0$-bounded case, one can apply the same reasoning, modulo the observation that every exact algebra in a variety of $0$-bounded hoops has a homomorphism to the 2-element Boolean algebra, and therefore so does any direct product of them. Indeed, every exact algebra embeds into some free algebra by definition, and the $0$-generated free algebra is a quotient of any other free algebra; in any nontrivial variety of $0$-bounded hoops $\vv V$, $\alg F_{\vv V}(\emptyset)$ coincides with the 2-element Boolean algebra. This concludes the proof.
\end{proof}
Some of the most relevant cases, besides Boolean and G\"odel algebras, are summarized in the following corollary. 
\begin{corollary}\label{cor:locallyfin}
	The following varieties, and the corresponding many-valued logics, have unitary e-generalization type: all subvarieties of G\"odel algebras; all locally finite subvarieties of MV-algebras and BL-algebras; all subvarieties of G\"odel hoops; all locally finite subvarieties of Wajsberg hoops and basic hoops.\end{corollary}

\subsection{3-valued Kleene logic}\label{sec:kleene}
Let us now consider the 3-valued Kleene logic, also known as Kleene's ``strong logic of indeterminacy''; this is a generalization of classical logic which adds to the intended model of the logic a third indeterminate value to the truth and falsum constants \cite{Kleene09}.  Its algebraic semantics, the variety $\vv{KA}$  of Kleene algebras, is a subvariety of bounded distributive lattices with an involution $\neg$ \cite{Kalman}. In particular $\vv{KA}$ is axiomatized by $x \approx \neg\neg x$, {\em De Morgan identity} $x\wedge y \approx\neg(\neg x\lor \neg y)$, and {\em Kleene identity} $x\wedge \neg x\leq y\vee\neg y$. 
As it is well-known $\vv{KA}$  is generated by its standard model which is the  3-element algebra $\alg K_3=(\{0,a, 1\}, \wedge, \vee, \neg, 0, 1)$ where $0 < a < 1$ and  $\neg 0 = 1, \neg 1 = 0, \neg a = a$. 

The free $1$-generated algebra $\free_{\vv{KA}}(z)$ is 
depicted in Figure \ref{figKleene1}, together with its quotients: ${\bf 4}_{\vv{KA}}$ and ${\bf 2}_{\vv{KA}}$ (respectively, the 4-element and 2-element Boolean algebras seen as Kleene algebras), the Kleene chains ${\bf K}_{4}$ and ${\bf K}_{3}$, of four and three elements, and the trivial algebra.
\begin{figure}
\begin{center}
\begin{tikzpicture}

  \node [label=below:{${0}$}, label=below:{} ] (n1)  {} ;
  \node [above   of=n1,label=left:{${z\wedge\neg z}$}, label=below:{ }] (n2)  {} ;
    \node [above right  of=n2,label=right:{${ \neg z}$}, label=below:{ }] (n3)  {} ;
        \node [above left  of=n2,label=left:{${z}$}, label=below:{ }] (n4)  {} ;
           \node [above right  of=n4,label=left:{${z\vee \neg z}$}, label=below:{ }] (n5)  {} ;
                 \node [above   of=n5,label=above:{${1}$}, label=below:{ }] (n6)  {} ;

  \draw  (n1) -- (n2);
  \draw (n2) -- (n3);
\draw (n2)--(n4);
\draw (n5)--(n4);
\draw (n5)--(n6);
\draw (n5)--(n3);

  \draw [fill] (n1) circle [radius=.5mm];
  \draw [fill] (n2) circle [radius=.5mm];
  \draw [fill] (n3) circle [radius=.5mm];
  \draw [fill] (n4) circle [radius=.5mm];
    \draw [fill] (n5) circle [radius=.5mm];
    \draw [fill] (n6) circle [radius=.5mm]; 

\end{tikzpicture}
\begin{tikzpicture}

  \node [label=below:{}, label=below:{} ] (n1)  {} ;
  \node [above   of=n1,label=below:{$0$}, label=below:{ }] (n2)  {} ;
    \node [above right  of=n2,label=right:{}, label=below:{ }] (n3)  {} ;
        \node [above left  of=n2,label=left:{}, label=below:{ }] (n4)  {} ;
           \node [above right  of=n4,label=above:{${1}$}, label=below:{ }] (n5)  {} ;
                 \node [above   of=n5,label=right:{}, label=below:{ }] (n6)  {} ;

  \draw (n2) -- (n3);
\draw (n2)--(n4);
\draw (n5)--(n4);
\draw (n5)--(n3);

  \draw [fill] (n2) circle [radius=.5mm];
  \draw [fill] (n3) circle [radius=.5mm];
  \draw [fill] (n4) circle [radius=.5mm];
    \draw [fill] (n5) circle [radius=.5mm];

\end{tikzpicture}
\begin{tikzpicture}

  \node [label=below:{$0$}] (n1)  {} ;
  \node [above   of=n1,label=left:{}, label=below:{ }] (n2)  {} ;
    \node [above   of=n2,label=left:{}, label=below:{ }] (n3)  {} ;
        \node [above   of=n3,label=above:{${1}$}, label=below:{ }] (n4)  {} ;
           \node [above right  of=n4,label=left:{}, label=below:{ }] (n5)  {} ;
                 \node [above   of=n5,label=right:{}, label=below:{ }] (n6)  {} ;

  \draw  (n1) -- (n2);
  \draw (n2) -- (n3);
\draw (n3)--(n4);
  \draw [fill] (n1) circle [radius=.5mm];
  \draw [fill] (n2) circle [radius=.5mm];
  \draw [fill] (n3) circle [radius=.5mm];
  \draw [fill] (n4) circle [radius=.5mm];
 \end{tikzpicture}
\begin{tikzpicture}

  \node [label=below:{$0$}] (n1)  {} ;
  \node [above   of=n1,label=left:{$$}, label=below:{ }] (n2)  {} ;
    \node [above   of=n2,label=above:{${ 1}$}, label=below:{ }] (n3)  {} ;
           \node [above right  of=n3,label=left:{}, label=below:{ }] (n5)  {} ;
                 \node [above   of=n5,label=right:{}, label=below:{ }] (n6)  {} ;

  \draw  (n1) -- (n2);
  \draw (n2) -- (n3);
  \draw [fill] (n1) circle [radius=.5mm];
  \draw [fill] (n2) circle [radius=.5mm];
  \draw [fill] (n3) circle [radius=.5mm];
  \end{tikzpicture}
\begin{tikzpicture}
  \node [label=below:{$0$}] (n1)  {} ;
  \node [above   of=n1,label=above:{${ 1}$}, label=below:{ }] (n2)  {} ;
    \node [above   of=n2,label=left:{}, label=below:{ }] (n3)  {} ;
           \node [above right  of=n3,label=left:{}, label=below:{ }] (n5)  {} ;
                 \node [above   of=n5,label=right:{}, label=below:{ }] (n6)  {} ;
  \draw  (n1) -- (n2);
  \draw [fill] (n1) circle [radius=.5mm];
  \draw [fill] (n2) circle [radius=.5mm];
\end{tikzpicture}
\begin{tikzpicture}
  \node [label=below:{$0=1$}] (n1)  {} ;
  \draw [fill] (n1) circle [radius=.5mm];
\end{tikzpicture}
\end{center}
\caption{On the left-most, the free 1-generated Kleene algebra $\free_{\vv{KA}}(z)$; then from left-to-right its quotients 
${\bf 4}_{\vv{KA}}$ (generated by the congruence $\langle(z\vee \neg z), 1\rangle$), ${\bf K}_{4}$ (given by $\langle z, (z\vee \neg z)\rangle$ or $\langle z, (z\wedge \neg z)\rangle$), ${\bf K}_{3}$ (obtained by $\langle z, \neg z\rangle$), ${\bf 2}_{\vv{KA}}$ (from $\langle z,1\rangle$ $\langle \neg z,1\rangle$), and the trivial algebra (by $\langle 0,1\rangle$). }\label{figKleene1}
\end{figure}
For the next results we use the following characterization of exact Kleene algebras in \cite{CM15}:
\begin{lemma}[{\cite[Lemma 34]{CM15}}]\label{LemmaKleeneCM}
A Kleene algebra is exact iff it is non-trivial, $1$ is join irreducible, and it satisfies the quasi-equation
\begin{equation}\label{qeqCM15}
 \neg x \leq x, x \land \neg y \leq \neg x \lor y \mbox{ implies } \neg y \leq y.
 \end{equation}
\end{lemma}
We observe that if a Kleene algebra has a negation fixpoint $f=\neg f$ it is not exact, because the above quasi-equation (\ref{qeqCM15}) fails. Indeed, $\neg f\leq f$ and  $f=f\wedge \neg 0\leq \neg f\vee 0=\neg f$ but $\neg 0=1>0$. 
\begin{proposition}\label{KleeneExact}
The 1-generated exact Kleene algebras are $\free_{\vv{KA}}(z)$,  ${\bf K}_{4}$, and ${\bf 2}_{\vv{KA}}$, and they are all projective. 
\end{proposition}
\begin{proof}
Let us show that, among the 1-generated algebras depicted in Figure \ref{figKleene1}, only $\free_{\vv{KA}}(z)$,  ${\bf K}_{4}$, ${\bf 2}_{\vv{KA}}$ are exact. First, $\free_{\vv{KA}}(z)$ is a retract of itself and it is hence exact  and projective. It is also clear that ${\bf 2}_{\vv{KA}}$ is a retract of $\free_{\vv{KA}}(z)$ and therefore it is projective and hence exact.

The same argument holds for ${\bf K}_{4}$. Indeed, naming its lattice reduct $0 < \neg x < x < 1$, consider the embedding into $\free_{\vv{KA}}(z)$ via the map $i$ such that $i(1)=1$, $i(x)=z\vee \neg z$, $i(\neg x)=z\wedge\neg z$, $i(0)=0$. The surjective homomorphism $j:\free_{\vv{KA}}(z)\to {\bf K}_{4}$ defined by $j(z) = j(z \lor \neg z) = x$ and $j(\neg z) = j(z \land \neg z) = \neg x$, composed with $i$ gives the identity on ${\bf K}_{4}$, which is therefore a retract of the free algebra hence projective. 

It is left to prove that the other 1-generated algebras are not exact. The trivial algebra is obviously not isomorphic to a subalgebra of any non-trivial algebra and therefore it cannot be exact. As for $\alg 4_{\vv{KA}}$, $1$ is not join irreducible and therefore the claim follows by Lemma \ref{LemmaKleeneCM}. Finally, ${\alg 3}_{\vv{KA}}$ is not exact since it has a negation fixpoint and, as discussed above, this implies a failure of identity (\ref{qeqCM15}) which is necessary for exactness by Lemma \ref{LemmaKleeneCM}.
\end{proof}
It follows in particular that $\vv{KA}$ is 1EP; we demonstrate that the stronger property 1ESP also holds. To this end, we make use of  the duality between  Kleene algebras and particular posets with involution  from \cite{CF77} and its adaptation presented in  \cite{BC13}.  Since we are only interested in finitely generated algebras and $\vv{KA}$ is locally finite, we consider the restriction of the duality to finite structures. 
\begin{definition}
	Let us call $\mathcal{PK}_f$ the category of {\em finite involutive posets} where:
\begin{enumerate}
\item The objects are structures $(P, \leq, \iota)$ where $(P,\leq)$ is a finite poset, and $\iota: P\to P$ is such that $x \leq y$ implies $\iota(y) \leq \iota(x)$,  $\iota(\iota(x)) = x$ and every $x$ is order comparable with $\iota(x)$, i.e. either $x\leq \iota(x)$ or $\iota(x)\leq x$.
\item The morphisms  are maps $f : (P,\leq,\iota) \to (P',\leq ',\iota')$ such that $x \leq y$ implies $f(x) \leq' f(y)$ and $f(\iota(x))=\iota'(f(x))$.
\end{enumerate}
\end{definition}

The functors implementing the duality extend those of Priestley duality \cite{Pr1,Pr2}. On objects, the functor maps a Kleene algebra $\alg A$ to the poset $P$ of its join irreducible elements, and the involution is defined as follows: for all $x\in P$, 
\begin{equation}\label{eqInv}
\iota(x)=\bigwedge A-\{\neg a: x\leq_{\alg A} a\}.
\end{equation}
In Figure \ref{figKleene3} we illustrate the duals of the 1-generated exact Kleene algebras.
\begin{figure}
\begin{center}
\begin{tikzpicture}

  \node [label=below:{}, label=below:{} ] (n1)  {} ;
  \node [above   of=n1,label=below:{${z\wedge\neg z}$}, label=below:{ }] (n2)  {} ;
    \node [above right  of=n2,label=right:{${ \neg z}$}, label=below:{ }] (n3)  {} ;
        \node [above left  of=n2,label=left:{${z}$}, label=below:{ }] (n4)  {} ;
           \node [above right  of=n4,label=above:{${1}$}, label=below:{ }] (n5)  {} ;
                 \node [above   of=n5,label=right:{}, label=below:{ }] (n6)  {} ;

  \draw (n2) -- (n3);
\draw (n2)--(n4);
\draw (n5)--(n4);
\draw (n5)--(n3);
\draw[->](n2) [bend left=25] to (n5);
\draw[->](n5) [bend left=25] to (n2);
      \draw [->] (n4) arc [radius=3mm, start angle=0, end angle= 340]  (n4);
            \draw [->] (n3) arc [radius=4mm, start angle=180, end angle= 510]  (n3);

  \draw [fill] (n2) circle [radius=.5mm];
  \draw [fill] (n3) circle [radius=.5mm];
  \draw [fill] (n4) circle [radius=.5mm];
    \draw [fill] (n5) circle [radius=.5mm];

\end{tikzpicture}
\hspace{.5cm}
\begin{tikzpicture}

  \node [label=below:] (n1)  {} ;
  \node [above   of=n1,label=below:{${\neg z}$}, label=below:{ }] (n2)  {} ;
    \node [above   of=n2,label=left:{${ z}$}, label=below:{ }] (n3)  {} ;
        \node [above   of=n3,label=above:{${1}$}, label=below:{ }] (n4)  {} ;
           \node [above right  of=n4,label=left:{}, label=below:{ }] (n5)  {} ;
                 \node [above   of=n5,label=right:{}, label=below:{ }] (n6)  {} ;

  \draw (n2) -- (n3);
\draw (n3)--(n4);
\draw[->](n2) [bend left=45] to (n4);
\draw[->](n4) [bend left=45] to (n2);
       \draw [->] (n3) arc [radius=2mm, start angle=180, end angle= 510]  (n3);
  \draw [fill] (n2) circle [radius=.5mm];
  \draw [fill] (n3) circle [radius=.5mm];
  \draw [fill] (n4) circle [radius=.5mm];
 \end{tikzpicture}
\hspace{.5cm}
\begin{tikzpicture}
  \node [label=below:{}] (n1)  {} ;
  \node [above   of=n1,label=below:{${ 1}$}, label=below:{ }] (n2)  {} ;
    \node [above   of=n2,label=left:{}, label=below:{ }] (n3)  {} ;
           \node [above right  of=n3,label=left:{}, label=below:{ }] (n5)  {} ;
                 \node [above   of=n5,label=right:{}, label=below:{ }] (n6)  {} ;
       \draw [->] (n2) arc [radius=4mm, start angle=270, end angle= 610]  (n2);
  \draw [fill] (n2) circle [radius=.5mm];
\end{tikzpicture}
\end{center}
\caption{The involutive posets dual of the 1-generated exact Kleene algebras $\free_{\vv{KA}}(1)$, ${\bf K}_{4}$ and ${\bf 2}_{\vv{KA}}$ respectively.}\label{figKleene3}
\end{figure}
The duals of projective Kleene algebras are characterized in \cite{BC13} as follows.
\begin{theorem}[{\cite[Theorem 12]{BC13}}]\label{thmProjKleene}
A Kleene algebra $\alg A$ is projective in $\vv{KA}$ iff its dual poset $(P,\leq ,\iota)$ in $\mathcal{PK}_f$ is such that:
\begin{enumerate}
\item\label{M2} For all $x\in P$, if $x\leq \iota(x)$ then there exists $y\in P$ such that $x\leq y=\iota(y)$;
\item\label{M3} $\{x\in P: x\leq \iota(x)\}$ is 3-complete\footnote{A poset $(Q,\leq)$ is 3-complete iff given any subset $X$ of $Q$, if the upper bound of every pair of elements in $X$ exists, then $\bigvee X$ exists in $(Q,\leq)$.}.
\item \label{K1}$\{x\in P: x\leq \iota(x)\}$ is a non-empty meet-semilattice.
\item\label{K2} For all $x,y\in P$ such that $x, y\leq \iota(x), \iota(y)$ there is a common upper bound $z\in P$ such that $z\leq \iota(z)$.
\end{enumerate}
\end{theorem}
We  denote by $\mathcal{PK}_{f}^p$ the set of posets that are dual to finite projective algebras in $\vv{KA}$.
\begin{theorem}\label{thmKleene1ESP}
$\vv{KA}$ has the 1ESP property.
\end{theorem}
\begin{proof}
By Proposition \ref{KleeneExact} we need to check that $\free_{\vv{KA}}(1)$, ${\bf K}_{4}$, and ${\bf 2}_{\vv{KA}}$ are strongly projective. Using the above duality a projective algebra $\alg A$ is strongly projective iff its corresponding poset $(Q,\leq, \iota)$ is such that: whenever there is a surjective morphism $\varphi:(P, \leq', \iota')\to (Q,\leq ,\iota)$ with $(P, \leq', \iota')$ the dual of another finitely generated projective algebra, there is $\psi: (Q, \leq, \iota)\to (P,\leq' ,\iota')$ such that $\varphi\circ\psi$ is the identity on $(Q,\leq, \iota)$. 

Let us fix an arbitrary $\mathcal{P}=(P,\leq, \iota)\in \mathcal{PK}_{f}^p$ that will be used in the rest of the proof. We start analyzing the case of ${\bf 2}_{\vv{KA}}$ whose dual is the 1-element involutive poset $\mathcal{U}=(\{1\}, =, id)$. Suppose there exists a surjective morphism $\varphi: \mathcal{P}\to \mathcal{U}$. Fix  $x\in P$ then either $x\leq \iota(x)$ or $\iota(x)\leq x=\iota\iota(x)$. In either case, by (\ref{M2}) in Theorem \ref{thmProjKleene}, there exists $y\in P$ such that $y=\iota(y)$. Let then $\psi:\mathcal{U}\to \mathcal{P}$ be defined by $\psi(1)=y$. Thus, $\psi$ is a morphism such that $\varphi\circ\psi$ is the identity map. This shows that ${\bf 2}_{\vv{KA}}$ is strongly projective.

Now we consider ${\bf K}_{4}$, whose dual is the poset $\mathcal{Q}=(\{1, z, \neg z\}, \leq , \iota_Q)$ where $\neg z\leq z\leq 1$ and $\iota_Q(1)=\neg z$, $\iota_Q(z)=z$ and $\iota_Q(\neg z)=1$ as in Figure \ref{figKleene3}. Suppose there exists $\varphi: \mathcal{P}\to\mathcal{Q}$. Let us start considering any $b\in \varphi^{-1}(\neg z)$; then $b<\iota(b)$ since otherwise, it would be $\iota(b) \leq b$ and then 
$$
1=\iota_Q(\neg z)=\iota_Q(\varphi(b))=\varphi(\iota(b))\leq \varphi(b)=\neg z,
$$ 
a contradiction. Now, one can find an element $d$ that is an involution fixpoint and such that $\varphi(d)=z$. Indeed, consider any $c\in \varphi^{-1}(z)$; then either $c\leq \iota(c)$ or $\iota(c)\leq c = \iota\iota(c)$. Either way, by (\ref{M2}) in Theorem \ref{thmProjKleene}, there exists $d$ such that $d=\iota(d)$ therefore necessarily $\varphi(d)=z$ since it is the only negation fixpoint. Note that $d, b\in \{x\in P: x\leq \iota(x)\}$ which is a meet-semilattice by (\ref{K1}) in Theorem \ref{thmProjKleene}. Thus, $d\wedge b$ is such that $d\wedge b\leq \iota(d\wedge b)$. Moreover, $d\wedge b\leq d$ implies $d=\iota(d)\leq \iota(d\wedge b)$. Therefore, in $\mathcal{P}$, $d\wedge b\leq d\leq \iota(d\wedge b)$. Let us check the $\varphi$-images of these elements. By order-preservation, $\varphi(d\wedge b)\leq \varphi(b)=\neg z$. Therefore, also $\varphi(d\wedge b)=\neg z$, and since $d\wedge b\leq b$ and $\iota$ is order-reversing, we get that $\iota(b)\leq \iota(d\wedge b)$. Hence 
$$
1= \iota_Q(\neg z) = \iota_Q(\varphi(b)) =\varphi(\iota(b))\leq \varphi(\iota(d\wedge b)),
$$ 
whence $\varphi(\iota(d\wedge b))=1$. We can then define the map $\psi:\mathcal{Q}\to\mathcal{P}$ in the following way:
$$
\psi(\neg z)=d\wedge b;\;\; \psi(z)=d;\;\; \psi(1)=\iota(d\wedge b).
$$
Such a $\psi$ is order-preserving since $d\wedge b\leq d\leq \iota(d\wedge b)$ and it commutes with $\iota$ by direct computation. Moreover, $\varphi\circ \psi$ is the identity since we have seen that $\varphi(d\wedge b)=\neg z$; $\varphi(d)=z$ and $\varphi(\iota(d\wedge b))=1$. Therefore also ${\bf K}_{4}$ is strongly projective.

Finally, let us consider the case of $\free_{\vv{KA}}(1)$ whose dual poset is $\mathcal{D}=(\{1, z, \neg z, z\wedge \neg z\},\leq, \iota_D)$ where, in particular, $\iota_D(z)=z$; $\iota_D(\neg z)=\neg z$; $\iota_D(1)=z\wedge \neg z$ and $\iota_D(z\wedge \neg z)=1$ (see Figure \ref{figKleene3}). Let $\varphi: \mathcal{P}\to \mathcal{D}$. Using (\ref{M2}) in Theorem \ref{thmProjKleene}, one can find $a, b\in P$ such that $a=\iota(a)$, $b=\iota(b)$ and $\varphi(a)=z$, $\varphi(b)=\neg z$. Now, $a, b\in \{x\in P: x\leq \iota(x)\}$ and therefore by (\ref{K1}) in Theorem \ref{thmProjKleene} there exists $a\wedge b$ and it is such that $a\wedge b\leq \iota(a\wedge b)$. Note that, by order-preservation, $\varphi(a\wedge b)\leq \varphi(a)=z$, $\varphi(a\wedge b)\leq \varphi(b)=\neg z$. Therefore, necessarily, $\varphi(a\wedge b)=z\wedge \neg z$. Moreover, $\varphi(\iota(a\wedge b))=\iota_D(\varphi(a\wedge b))=\iota_D(z\wedge \neg z)=1$. Finally note that $a\wedge b\leq a, b$ implies $a=\iota(a)\leq \iota(a\wedge b)$ and similarly $b=\iota(b)\leq \iota(a\wedge b)$. Hence, define $\psi:\mathcal{D}\to\mathcal{P}$ to be the map
$$
\psi(z)=a; \;\;\psi(\neg z)=b;\;\; \psi(1)=\iota(a\wedge b);\;\; \psi(z\wedge \neg z)=a\wedge b.
$$
Then, $\psi$ preserves the order because we have shown that both $a, b\leq \iota(a\wedge b)$ and of course $a\wedge b\leq a, b$. Moreover, $\psi$ commutes with $\iota$ as it can be verified by direct computation and the composition with $\varphi$ is the identity by construction. Thus $\psi$ is a morphism that witnesses the fact that $\free_{\vv{KA}}(1)$ is strongly projective. This concludes the proof.
\end{proof}
By Theorem \ref{thm1ESP} and Corollary \ref{cor:anti1esp} we can hence study the e-generalization type via the lattice of congruences of $\free_{\vv{KA}}(z)$ 
which is is depicted  in Figure \ref{figKleene2},  with highlighted the projective congruences.

\begin{figure}
\begin{center}

\begin{tikzpicture}

  \node [label=below:{\tiny$\Delta$}] (n1)  {} ;
  \node [above   of=n1,label=below:{\tiny${\langle z\vee \neg z, 1 \rangle}$}, label=below:{ }] (n2)  {} ;
    \node [left   of=n2,label=left:{\tiny${\langle z, z\vee\neg z \rangle}$}, label=below:{ }] (n3)  {} ;
        \node [right   of=n2,label=right:{\tiny${\langle z, z\wedge \neg z \rangle}$}, label=below:{ }] (n4)  {} ;
                \node [above   of=n2,label=above:{\tiny${\langle z,\neg z \rangle}$}, label=below:{ }] (n5)  {} ;
                \node [above   of=n3,label=left:{\tiny${\langle z,1 \rangle}$}, label=below:{ }] (n6)  {} ;
                \node [above   of=n4,label=right:{\tiny${\langle z,0 \rangle}$}, label=below:{ }] (n7)  {} ;
                  \node [above   of=n5,label=above:{\tiny${\langle 0, 1 \rangle}$}, label=below:{ }] (n8)  {} ;

  \draw  (n1) -- (n2);
  \draw (n1) -- (n3);
\draw (n1)--(n4);
\draw (n3)--(n5);
\draw (n3)--(n6);
\draw (n2)--(n6);
\draw (n2)--(n7);
\draw (n4)--(n5);
\draw (n4)--(n7);
\draw (n5)--(n8);
\draw (n6)--(n8);
\draw (n7)--(n8);
  \draw [fill] (n1) circle [radius=.5mm];
  \draw [fill] (n2) circle [radius=.5mm];
  \draw [fill] (n3) circle [radius=.5mm];
  \draw [fill] (n4) circle [radius=.5mm];
    \draw [fill] (n5) circle [radius=.5mm];
      \draw [fill] (n6) circle [radius=.5mm];
        \draw [fill] (n7) circle [radius=.5mm];
          \draw [fill] (n8) circle [radius=.5mm];
             \draw  (n1) circle [radius=.8mm];
                \draw  (n3) circle [radius=.8mm];
                   \draw  (n4) circle [radius=.8mm];
                      \draw  (n6) circle [radius=.8mm];
                         \draw  (n7) circle [radius=.8mm];
 \end{tikzpicture}

\end{center}
\caption{The lattice of congruences of $\free_{\vv{KA}}(z)$,  with highlighted the projective congruences.}\label{figKleene2}
\end{figure}

\begin{theorem}
The e-generalization type of $\vv{KA}$ and 3-valued Kleene logic is unitary. 
\end{theorem}
\begin{proof}
Since $\vv{KA}$ has the 1ESP property, by Corollary \ref{cor:anti1esp} the type of any algebraic e-generalization problem $h: \free_{\vv{KA}}(z) \to \prod_{k=1}^m \alg E_k$ is the type of $\Gen(h)$, which consists of the cardinality of a complete maximal set of projective congruences below $\ker(h)$. By Theorem \ref{propThetat} the E-congruences of $\free_{\vv{KA}}(z)$, which can appear as $\ker(h)$, are the finite intersection of projective ones; by direct inspection of the lattice of congruences of $\free_{\vv{KA}}(z)$ depicted in Figure \ref{figKleene2}, the E-congruences are all the projective ones plus the one given by $\langle z\vee\neg z, 1\rangle$ (which only appears for $m \geq 2$, while the other ones appear for problems with any $m \geq 1$). The other congruences are neither exact nor intersection of exact ones. Therefore there is always one maximal projective congruence in $\Gen(h)$ (in fact, in all cases except one, $\ker(h)$ is projective itself), which yields unitary type of $\Gen(h)$ and consequently of every e-generalization problem $h$.
\end{proof}

\subsection{\luk\ logic}\label{sec:luk}
\luk\ infinite-valued logic is one of the most well-studied many-valued logics; its equivalent algebraic semantics is the variety $\vv{MV}$ of MV-algebras. The latter is generated by the {\em standard MV-algebra} over the real unit interval $[0,1]$, in the language $\{\oplus, \neg, 0 , 1\}$, and operations defined as $x \oplus y = \min\{x + y, 1\}$, $\neg x = 1 - x$.
MV-algebras enjoy a deep connection with geometrical objects, which has proven extremely fruitful to study logical properties of \luk\ logic. Precisely, {\em rational polyhedra}. Next, we will  recall the basic geometrical and topological facts needed to follow the main results we will prove in this section. For the unexplained notions we suggest to consult  \cite{Armstrong} as a general reference, and  \cite{MS13}, together with \cite{CabrerForum,Mundicibook}, for specific results on polyhedral geometry.

The finitely presented members of $\vv{MV}$ form a category that is a full subcategory of $\vv{MV}$ itself, whence its morphisms are homomorphisms between finitely presented MV-algebras. 
A {\em rational polytope} of $[0,1]^m$ is the convex hull of finitely many points $q_1,\ldots, q_n\in [0,1]^m$ having rational coordinates. A {\em rational polyhedron} $\mathscr{R}$ of $[0,1]^m$ is the (set theoretic) union of finitely many rational polytopes of $[0,1]^m$. A function $\zeta$ between two rational polyhedra $\mathscr{R}_1\subseteq[0,1]^m$ and $\mathscr{R}_2\subseteq[0,1]^k$ is said to be a {\em $\mathbb{Z}$-map} if there are functions $\zeta_1,\ldots, \zeta_k:[0,1]^m\to[0,1]$ that are continuous, piecewise linear, and with each piece having only integer coefficients, and such that
$$
\zeta: (q_1,\ldots, q_m)\in \mathscr{R}_1\mapsto (\zeta_1(q_1,\ldots, q_m), \ldots, \zeta_k(q_1,\ldots, q_m))\in \mathscr{R}_2.
$$
Conforming to the standard notation, we will call {\em McNaughton functions} (cf. \cite{McN}) those $\mathbb{Z}$-maps $\zeta:[0,1]^m\to[0,1]$, for some $m \geq 1$.

Rational polyhedra with $\mathbb{Z}$-maps form the category $\mathcal{RP}$ that turns out to be dual to that of finitely presented MV-algebras, \cite[Theorem 3.4]{MS13}; in particular, $n$-generated finitely presented algebras correspond to polyhedra in $[0,1]^n$. We will henceforth refer to this duality as {\em Marra-Spada duality} by the names of the authors of \cite{MS13}.

We also note that monomorphisms and epimorphisms in $\mathcal{RP}$ are, respectively, injective and surjective $\mathbb{Z}$-maps \cite[Theorem 3.2]{CabrerForum}.  Following \cite{Mundicibook}, a map $\zeta$ between two rational polyhedra $\mathscr{R}_1$ and $\mathscr{R}_2$ is a {\em $\mathbb{Z}$-homeomorphism} if it is a bijective $\mathbb{Z}$-map and its inverse is also a $\mathbb{Z}$-map.
Moreover, surjective homomorphisms on the algebraic side correspond to {\em strict $\Z$-maps} on the geometric side, that is to say, $\Z$-maps that are $\Z$-homeomorphisms onto their range \cite[Theorem 3.5]{CabrerForum}.

Therefore, one can translate generalization problems and their solutions in the dual setting. It will be convenient to use the alternative presentation outlined in Remarks \ref{remark:equivalentpresentation} and \ref{remark:freeinsteadofprojective}. We then consider an algebraic problem to be defined by a tuple of surjective homomorphisms $\{h_k\}_{k = 1}^m: \alg F_{\vv{MV}}(z) \to \alg E_k$, and a solution to be a homomorphism to some finitely generated free algebra $g: \alg F_{\vv{MV}}(z) \to \alg F_{\vv{MV}}(Y)$, testified by a tuple $\{f_k\}_{k = 1}^m$ such that $f_k: \alg F_{\vv{MV}}(Y) \to \alg E_k$ and $h_k = f_k \circ g$. 

On the dual side, the above presentation yields the following geometric reading. A problem is going to be represented by a family of strict $\Z$-maps $\eta_k: \scr E_k \to [0,1]$ where each $\scr E_k \sse [0,1]$ is the dual of a 1-generated exact algebra, for $k = 1 \ldots m$, while, as recalled above, $[0,1]$ is the dual of $\alg F_{\vv{MV}}(z)$. A solution is going to be a $\Z$-map $\gamma: [0,1]^n \to [0,1]$ for some $n \in \mathbb{N}$, such that there exists a family $\varphi_k: \scr E_k \to [0,1]^n$ of strict $\Z$-maps such that $\eta_k = \gamma \circ \varphi_k$.

The generality order also obviously translates to this setting. In particular, we write $\gamma \sqsubseteq \gamma'$ if and only if there is $\alpha$ such that $\gamma = \gamma' \circ \alpha$.

The proposition below entails that $\vv{MV}$ has the 1EP property. 
\begin{proposition}\label{propMVInter}
1-generated exact MV-algebras coincide with 1-generated projective MV-algebras. Moreover, their dual polyhedra are exactly the closed intervals of the real valued unit interval $[0,1]$ of the kind $[0, q]$ or $[r, 1]$ with $q, r\in [0,1]\cap\mathbb{Q}$.
\end{proposition}
\begin{proof}
For the first claim suffices to show that 1-generated exact implies projective since all projective algebras are exact. Since finitely generated exact  MV-algebras are finitely presented (see \cite[Lemma 2.8]{CabrerForum}), we can apply the Marra-Spada duality, and consider the corresponding dual polyhedron $\mathscr{P}$ in $[0,1]$.  Now, by \cite[Lemma 4.12]{CabrerForum}, dual polyhedra of exact MV-algebras are connected and contain a point whose coordinates are all in $\{0,1\}$; this means that $\mathscr{P}$ is a connected subinterval of $[0,1]$ that contains either $0$ or $1$. Thus, it is necessarily of the kind $[0, q]$ or $[r, 1]$ with $q, r\in [0,1]\cap\mathbb{Q}$. These are exactly the 1-generated projective MV-algebras as shown in \cite[Lemma 4.2]{MS13}. This proves both claims.
\end{proof}
However, $\vv{MV}$ does not have the 1ESP property as the following example shows.

\begin{example}
Let us consider the two (closed) rational intervals $[0,1/4]$ and $[0, 1/2]$ which are duals to projective MV-algebras denoted $\alg P_{1/4}$ and $\alg P_{1/2}$. There is a surjective $\mathbb{Z}$-map (i.e., an epimorphism) $\varphi:[0,1/4]\to[0,1/2]$ given by $\varphi(x)=2x$. Notice that, however, there is no $\mathbb{Z}$-map $\psi:[0,1/2]\to[0,1/4]$ such that $\varphi\circ \psi$ is the identity map. Indeed, if there were such a $\psi$, it would hold in particular that $\varphi(\psi(1/2))=1/2$, i.e., necessarily $\psi(1/2)=1/4$. However, since $\mathbb{Z}$-maps are linear  with integer coefficients, such $\psi$ cannot exist, given that necessarily $\psi(1/2) \in \{0, 1/2, 1\}$. Via  duality, this implies that while there is an embedding (i.e., a monomorphism) from $h:\alg P_{1/2}\to\alg P_{1/4}$, there is no homomorphism $g:\alg P_{1/4}\to\alg P_{1/2}$ such that $g\circ h$ is the identity. This provides a failure of the 1ESP.
\hfill $\Box$ \end{example}

Even without the 1ESP, we will see that in many cases one can restrict to consider geometric solutions with domain $[0,1]$, equivalently, symbolic solutions of one variable, in order to study the e-generalization type of a problem.

Let us build some intuition about e-generalization problems in MV-algebras. Note that $\vv{MV}$, in contrast with the previous varieties under consideration, is not locally finite, and actually the 1-generated free algebra is isomorphic to an infinite countable algebra of continuous function (cf. \cite{McN,Mundicibook}), with a complicated lattice order. However, by Corollary \ref{thm1ESP}, the order between projective congruences is enough to give a rough description of the poset of solutions of a fixed e-generalization problem. 
In order to describe this order, we prepare.

We call $\mathcal{PI}$ the set of intervals of the kind $[0, q]$ or $[r, 1]$ with $q, r\in [0,1]\cap\mathbb{Q}$ which, by Proposition \ref{propMVInter}, correspond to 1-generated exact or projective MV-algebras.
By the duality, the poset of exact congruences is isomorphic to the poset on $\mathcal{PI}$ ordered by reverse inclusion\footnote{The reader may check Section 3.4 in \cite{Cignolietal} for the connection between congruences of free MV-algebras and rational polyhedra, which arise as $0$-sets of McNaughton functions. In particular, the desired dual order isomorphism follows from Lemma 3.4.8 therein.} and hence it looks as in the left-hand side of Figure \ref{fig:MV1gen}. Moreover, since the kernels of generalization problems (i.e., E-congruences) are finite intersections of exact congruences (Theorem \ref{propThetat}), which correspond to the union of the corresponding polyhedra, we see their poset on the right-hand side of Figure \ref{fig:MV1gen}.
\begin{figure}
\begin{center}
$\xymatrix{
\{0\}\ar@{--}[d] && \{1\}\ar@{--}[d]\\
[0, q] \ar@{--}[dr]&& [r,1]\ar@{--}[dl]\\
&[0,1]&
}$\hspace{2cm}	
$\xymatrix{
\{0\}\ar@{--}[d]\ar@{--}[dr] && \{1\}\ar@{--}[d]\ar@{--}[dl]\\
[0, q] \ar@{--}[dr]\ar@{--}[ddr]& \{0\} \ar@{--}[d]\cup \{1\} & [r,1]\ar@{--}[ddl]\ar@{--}[dl]\\
&[0,q] \cup [r,1] \ar@{--}[d]&\\
&[0,1]&
}
$	
\end{center}
\caption{On the left-hand side, the poset of exact congruences of $\alg F_{\vv{MV}}(z)$. On the right, the poset of E-congruences for $\vv{MV}$. Given a 1-generated exact congruence $\theta$, it is labeled in the posets above by the rational polyhedron associated via the duality to the finitely presented MV-algebra $\alg F_{\vv{MV}}(z)/\theta$.}\label{fig:MV1gen}
\end{figure}

Precisely, consider a geometric problem, represented by a family of strict $\Z$-maps $\{\eta_k\}_{k = 1}^m$, with $\eta_k: \scr E_k \to [0,1]$ and  $\scr E_k \sse [0,1]$ being as in Proposition \ref{propMVInter} 
for $k = 1, \ldots, m$.  To ease the reading we will henceforth name {\em exact intervals} the intervals like the $\scr E_k$'s that are dual to $1$-generated exact algebras. In the poset on the right of Figure \ref{fig:MV1gen}, the problem $\{\eta_k\}_{k = 1}^m$ is associated to the polyhedron $$\cc P(\{\eta_k\}_{k = 1}^m) = \bigcup_{k = 1}^m\eta_k[\scr E_k] \sse [0,1],$$ given by the union of the exact intervals determined over $[0,1]$ by the images of the maps $\eta_k$.

We are ready to show that, in many cases, the type of a solution can be fully determined by restricting to solutions to the 1-generated algebra, or, from the symbolic side, to one-variable solutions.

\begin{lemma}\label{prop:MVlessvar}
	Given a problem $\{\eta_k\}_{k = 1}^m$ such that $\cc P(\{\eta_k\}_{k = 1}^m) \neq [0,1]$, and a solution $\gamma: [0,1]^n \to [0,1]$, there is a solution $\gamma': [0,1]\to [0,1]$ that is less general than $\gamma$.
\end{lemma}
\begin{proof}
	Consider a problem $\{\eta_k\}_{k = 1}^m$ and a solution $\gamma: [0,1]^n \to [0,1]$, testified by the strict $\Z$-maps $\{\varphi_k\}_{k = 1}^m$, i.e., $\eta_k = \gamma \circ \varphi_k$. 
Notice that since all maps $\varphi_k$ and $\eta_k$ are strict $\Z$-maps, by definition they are $\Z$-homeomorphisms onto their images. 

Moreover, since $\cc P(\{\eta_k\}_{k = 1}^m) \neq [0,1]$ by hypothesis, $\eta_k[\scr E_k]$ includes either $\{0\}$ or $\{1\}$ for each $k \in \{1, \ldots, m\}$. Let us assume that $\cc P(\{\eta_k\}_{k = 1}^m) = \bigcup_{k = 1}^m\eta_k[\scr E_k]$ includes both $\{0\}$ and $\{1\}$, since the following proof can be easily adapted to the case where one of the two extrema is not included.
Let now $$H_0 = \{k \in \{1, \ldots, m\}: 0 \in \eta_k[\scr E_k]\}, \;\;H_1 = \{k \in \{1, \ldots, m\}: 1 \in \eta_k[\scr E_k]\}.$$
Then $H_0$ and $H_1$ partition $\{1, \ldots, m\}$. Given the finite number of maps, there is some $k_0 \in H_0$ such that 
$$\eta_{k_0}[\scr E_{k_0}] =\bigcup_{k \in H_0}\eta_k[\scr E_k].$$
Similarly, there is some $k_1 \in H_1$ such that $$\eta_{k_1}[\scr E_{k_1}] =\bigcup_{k \in H_1}\eta_k[\scr E_k].$$

Notice that since $\cc P(\{\eta_k\}_{k = 1}^m) \neq [0,1]$, then $\varphi_{k_0}[\scr E_{k_0}] \cap \varphi_{k_1}[\scr E_{k_1}] = \emptyset$ and $\varphi_{k_0}[\scr E_{k_0}] \cup \varphi_{k_1}[\scr E_{k_1}]$ is $\Z$-homeomorphic to $\eta_{k_0}[\scr E_{k_0}] \cup \eta_{k_1}[\scr E_{k_1}] = \cc P(\{\eta_k\}_{k = 1}^m)$ via $\gamma$.  Indeed, $\gamma\circ \varphi_k=\eta_k$ for all $k$ since $\gamma$ is a solution. Therefore,  there exists a $\Z$-map $\alpha$ which inverts $\gamma$ on the subset of its range given by $\gamma\left[\varphi_{k_0}[\scr E_{k_0}] \cup \varphi_{k_1}[\scr E_{k_1}]\right]$.
Precisely,  there exists $\alpha: \cc P(\{\eta_k\}_{k = 1}^m) \to [0,1]^n$,  which is a strict $\Z$-map such that for all $k = 1 \ldots m$,
\begin{equation}\label{eq:gammaprimo}\gamma \circ \alpha \circ \eta_k = \eta_k.\end{equation}

Now, since by definition a $\Z$-map is a certain collection of  McNaughton functions, we can extend $\alpha$ to a $\Z$-map defined over the whole unit interval $[0,1]$, $\hat\alpha: [0,1] \to [0,1]^n$ by considering the same McNaughton functions defining $\alpha$ over the whole domain (see \cite[Proposition 3.2]{Mundicibook} for details):
\begin{center}
$\xymatrix{
[0,1] \supseteq  \eta_{k_0}[\scr E_{k_0}] \cup \eta_{k_1}[\scr E_{k_1}] \ar@<-10ex>[dd]^-{\hat\alpha}  && \scr E_k \ar[ll]_-{\eta_k} \ar@<0.5ex>[lldd]^-{\varphi_k} \\
&\\
[0,1]^n \supseteq \varphi_{k_0}[\scr E_{k_0}] \cup \varphi_{k_1}[\scr E_{k_1}] \ar@<11ex>[uu]^-{\gamma}&
}$
\end{center}

The map $\gamma': [0,1] \to [0,1]$ defined by $\gamma': = \gamma \circ \hat\alpha$ then provides the desired solution. Indeed, it is a solution testified by the maps $\eta_k$, as a consequence of (\ref{eq:gammaprimo}). Moreover, by definition, $\gamma'$ is less general than $\gamma$.
\end{proof}
In the case where $\cc P(\{\eta_k\}_{k = 1}^m) = [0,1]$, the above reasoning cannot be applied. In more details, generally one cannot find a single $\Z$-map $\alpha$ that acts as the local inverse of the solution $\gamma$.
Nonetheless, we can prove that one can restrict to considering solutions to the two-generated free algebra to determine the e-generalization type, or equivalently, to solutions in two variables.
\begin{lemma}\label{prop:MVlessvar2}
	Given a problem $\{\eta_k\}_{k = 1}^m$ such that $\cc P(\{\eta_k\}_{k = 1}^m) = [0,1]$, and a solution $\gamma: [0,1]^n \to [0,1]$, there is a solution $\gamma': [0,1]^2\to [0,1]$ that is less general than $\gamma$. \end{lemma}
\begin{proof}
	The claim is obvious if $1 \leq n \leq 2$, so we assume that $n \geq 3$. 	
Consider then a problem $\{\eta_k\}_{k = 1}^m$ and a solution $\gamma: [0,1]^n \to [0,1]$, testified by the strict $\Z$-maps $\{\varphi_k\}_{k = 1}^m$. 
Notice that since all maps $\varphi_k$ and $\eta_k$ are strict $\Z$-maps, by definition they are a $\Z$-homeomorphism onto their images. 
Hence, each $\scr E_k$ is $\Z$-homeomorphic to both $\varphi_{k}[\scr E_{k}] \sse [0,1]^n$ and $\eta_{k}[\scr E_{k}] \sse [0,1]$. Then, let $\beta_k: \varphi_k[\scr E_{k}] \to \scr E_{k}$ be the $\Z$-homeomorphism inverse to $\varphi_k$. 
Moreover, notice that $\varphi_{k}[\scr E_{k}]$ is $\Z$-homeomorphic to $\eta_k[\scr E_{k}]$ via $\gamma$. It follows that for all $k \in \{1, \ldots, m\}$,  there is a $\Z$-homeomorphism $\alpha_k: \eta_{k}[\scr E_{k}] \to \varphi_{k}[\scr E_{k}]$ such that $\gamma \circ \alpha_k \circ \eta_k = \eta_k$, acting as a local inverse. The situation is depicted in the following diagram:
\begin{center}
$\xymatrix{
[0,1] \supseteq \eta_k[\scr E_k] \ar@<0.5ex>[d]^-{\alpha_k}  && \scr E_k \ar[ll]_-{\eta_k} \ar@<0.1ex>[lld]^-{\varphi_k} \\
[0,1]^n \supseteq \varphi_k[\scr E_k] \ar@<0.5ex>[u]^-{\gamma} \ar@<0.5ex>[urr]^-{\beta_k}&
}$
\end{center}

Notice now that for each $k \in \{1, \ldots, m\}$, $\eta_k[\scr E_k]$ includes $\{0\}$ or $\{1\}$ (or both). By hypothesis, recall that $\cc P(\{\eta_k\}_{k = 1}^m) = \bigcup_{k = 1}^m\eta_k[\scr E_k] = [0,1]$.
Let then $$K_0 = \{k \in \{1, \ldots, m\}: 0 \in \eta_k[\scr E_k], 1 \notin \eta_k[\scr E_k] \}, \;\;K_1 = \{k \in \{1, \ldots, m\}: 1 \in \eta_k[\scr E_k]\}.$$
Then $K_0$ and $K_1$ partition $\{1, \ldots, m\}$. If $K_0 \neq \emptyset$, given its finiteness, there is some $k_0 \in K_0$ such that 
$$\eta_{k_0}[\scr E_{k_0}] =\bigcup_{k \in K_0}\eta_k[\scr E_k].$$
Similarly, there is some $k_1 \in K_1$ such that $$\eta_{k_1}[\scr E_{k_1}] =\bigcup_{k \in K_1}\eta_k[\scr E_k].$$
	
For $i \in \{k_0, k_1\}$, let us extend the above map $\alpha_i$ to a $\Z$-map on the whole $[0,1]$, $\hat\alpha_i: [0,1] \to [0,1]^n$, by considering the same McNaughton functions defining $\alpha_i$ on the extended domain (see again \cite[Proposition 3.2]{Mundicibook} for details). Let $i^* = 0$ if $i = k_0$ and $i^* = 1$ if $i = k_1$.
Then for each $k \in K_{i^*}$, since $\eta_k[\scr E_k] \sse \eta_i[\scr E_i]$:
\begin{equation}\label{eq:MVlessvar1}
	\gamma \circ \hat \alpha_i \circ \eta_k = \eta_k.
\end{equation}
Let us now lift the inverse $\Z$-homeomorphisms $\varphi_i: \scr E_i \to [0,1]^n $ and $\beta_i: \varphi_{i}[\scr E_{i}] \to \scr E_i$ to corresponding maps between the cubes  $[0,1]^n$ and $[0,1]^2$.
Precisely, first note that one can define $$\bar \beta_i: \varphi_{i}[\scr E_{i}] \to \{i^*\} \times \scr E_{i} \sse [0,1]^2$$ by simply assigning $\bar\beta_i(x) = (i^*, \beta_i(x))$. Then we consider $\hat \beta_i: [0,1]^n \to [0,1]^2$ to be the $\Z$-map on the domain $[0,1]^n$ extending $\bar\beta_i$, using the same McNaughton functions defining $\bar \beta_i$. 
For the (local) inverse, let 
$$\bar\varphi_i: \{i^*\} \times \scr E_i  \to \varphi_{i}[\scr E_{i}]$$
by composing the first projection with $\varphi_i$, $\bar\varphi_i: = \varphi_i \circ \pi_i$.  Since the domains of the $\varphi_i$'s are disjoint polyhedra,  they can be extended to a $\Z$-map defined on their union. 
Then, using again \cite[Proposition 3.2]{Mundicibook}, there is a $\Z$-map $\delta$ extending both $\bar \varphi_{k_0}$ and $\bar \varphi_{k_1}$ on the unit square, in other words, a map 
$$\delta: [0,1]^2 \to [0,1]^n$$
such that $\delta$ coincides with $\bar\varphi_i$ over $\{i^*\} \times \scr E_i$ for $i \in \{k_0, k_1\}$.  
Thus $\delta \circ  \hat\beta_i$ is the identity over $\varphi_{i}[\scr E_{i}]$ for each $i \in \{k_0, k_1\}$. Therefore, for each $k \in K_{i^*}$, since $\hat \alpha_i \circ \eta_k \sse \varphi_i[\scr E_i]$ we get:
\begin{equation}\label{eq:MVlessvar2}
\delta \circ \hat\beta_i \circ \hat \alpha_i \circ \eta_k = \hat \alpha_i \circ \eta_k.
\end{equation}
Let then $\gamma':= \gamma \circ \delta$. We claim that it is a solution to $\{\eta_k\}_{k = 1}^m$. Indeed, let $k \in K_{i^*}$, set $\varphi_k':= \hat\beta_i \circ \hat \alpha_i \circ \eta_k$. 
The situation is clarified in the following diagram:
	\begin{center}
$\xymatrix{
[0,1]\ar@<0.5ex>[d]^-{\hat\alpha_i}  && \scr E_k \ar[ll]_-{\eta_k} \ar[lld]^-{\varphi_k} \ar[lldd]^-{\varphi_k'}\\
[0,1]^n \ar@<0.5ex>[u]^-{\gamma} \ar@<0.5ex>[d]^-{\hat\beta_i} &\\
[0,1]^2 \ar@<0.5ex>[u]^-{\delta} \ar@/_{-2.3pc}/[uu]^-{\gamma'} 
}$
\end{center}
We obtain, by the identities (\ref{eq:MVlessvar1}) and (\ref{eq:MVlessvar2}) above:
$$\gamma' \circ \varphi_k' = \gamma \circ \delta \circ \hat\beta_i \circ \hat \alpha_i \circ \eta_k = \gamma \circ \hat \alpha_i \circ \eta_k = \eta_k.$$
Hence, $\gamma'$ is a solution, and it is less general than $\gamma$ by definition. This completes the proof.
\end{proof}
As a consequence:
\begin{theorem}\label{coroll:MV1}
	The study of the e-generalization type of a problem in $\vv{MV}$ can be reduced to considering solutions in two variables. For problems whose geometric counterpart $\{\eta_k\}_{k = 1}^m$ is such that $\cc P(\{\eta_k\}_{k = 1}^m) \neq [0,1]$, one can restrict to one-variable solutions.
\end{theorem}

We now proceed to show that the e-generalization type of $\vv{MV}$, and of \L ukasiewicz infinite-valued logic, is nullary. Specifically, we show that the symbolic problem $\{0,1\}$,  given by the constant terms $0$ and $1$, is of nullary type.

From the geometric perspective, we then consider the exact polyhedra $\scr E_0 = \{0\}$ and $\scr E_1 = \{1\}$, and the $\Z$-maps:
\begin{eqnarray*}
	\eta_0: \scr E_0 \to [0,1]&:& \quad \eta_0(0) = 0;\\
	\eta_1: \scr E_1 \to [0,1]&:& \quad \eta_1(1) = 1.
\end{eqnarray*}
Therefore, we consider the geometric problem $G = \{\eta_0, \eta_1\}$. By Lemma \ref{prop:MVlessvar}, one can restrict to solutions which are maps in one variable in order to find the minimal ones, and the next result adopts this simplification.
\begin{lemma}\label{prop:MVnull}
Given any solution $\gamma:[0,1] \to [0,1]$ to the problem $G$, there is a solution $\gamma':[0,1] \to [0,1]$ strictly less general than $\gamma$.
\end{lemma}
\begin{proof}
	Let $\gamma:[0,1] \to [0,1]$ be a solution to the problem $G$, with testifying $\Z$-maps $\varphi_0, \varphi_1$, i.e., such that $\gamma \circ \varphi_i = \eta_i$ for $i \in \{0,1\}$. Notice that $\gamma$ is a McNaughton function in one variable, necessarily mapping the lattice points $\{0,1\}$ to {\em distinct} lattice points. Hence, in particular, $\gamma$ is neither the function with constant value $1$ nor $0$. Therefore, there exists a reduced fraction $\frac{p}{q} \in \mathbb{Q} \cap [0,1]$  such that $\gamma(\frac{p}{q}) \notin \{0,1\}$, with $p, q \in \mathbb{N}$. 
	
	Consider the $\Z$-map $\alpha: [0,1] \to [0,1]$ such that $\alpha: x \longmapsto \min\{1, qx\},$ and let $\gamma':= \gamma \circ \alpha$. Then it is easily seen that $\gamma'$ is also a solution to $G$, since $\gamma' \circ \varphi_i = \eta_i$ for $i \in \{0,1\}$, given that $\alpha$ is the identity map on the lattice points $0$ and $1$. 
	
	By definition, $\gamma'$ is less general than $\gamma$; we proceed to prove that $\gamma'$ is in fact strictly less general than $\gamma$. To this end, suppose by way of contradiction that there is a $\Z$-map $\beta: [0,1] \to [0,1]$ testifying the opposite, i.e., such that $\gamma = \gamma' \circ \beta$. Consider the point $\frac{p}{q}$. Then 
\begin{equation}\label{eqGammaProof}\gamma \left(\frac{p}{q}\right) = \gamma'\left(\beta\left(\frac{p}{q}\right)\right) = \gamma\left(\alpha\left(\beta\left(\frac{p}{q}\right)\right)\right).
\end{equation} Notice that $\beta$ at $\frac{p}{q}$  takes value as a McNaughton function, that is to say, a linear polynomial with an integer coefficient, hence necessarily $\beta(\frac{p}{q}) \in \{0, \frac{1}{q}, \ldots, \frac{q-1}{q}, 1\}$. Therefore, since $\alpha$ maps each point $x$ to $\min\{1, qx\}$, necessarily $\alpha(\beta(\frac{p}{q})) \in \{0,1\}$, and then, by (\ref{eqGammaProof}),   $\gamma(\frac{p}{q})=\gamma(\alpha(\beta(\frac{p}{q}))) \in \{0,1\}$. However, we picked $\frac{p}{q}$ so that $\gamma(\frac{p}{q}) \notin \{0,1\}$, a contradiction. Thus, there can be no such $\Z$-map $\beta$, which implies that $\gamma'$ is strictly less general than $\gamma$, which is an arbitrary solution to the problem $G$. The proof is complete.
\end{proof}
The propositions above entail the main result of this section.
\begin{corollary}
	The e-generalization type of \L ukasiewicz logic and MV-algebra is nullary. 
\end{corollary}
\begin{proof}
	The geometric problem G presented above has nullary type. Indeed, by Lemma \ref{prop:MVlessvar}, one can restrict to considering solutions $\gamma: [0,1] \to [0,1]$ to find a minimal one. Lemma \ref{prop:MVnull} then yields that there can be no minimal solution; in other words, the set of minimal solutions is empty. Therefore, the type of $G$, and hence of $\vv{MV}$ and \L ukasiewicz logic, is nullary.
\end{proof}
Notice that this is in contrast with the locally finite subvarieties of $\vv{MV}$, where the e-generalization type is unitary (Corollary \ref{cor:locallyfin}), but in analogy with the unification type of $\vv{MV}$, which is nullary \cite{MS13}.
\section{Conclusions}

The  algebraic approach we presented offers a novel perspective and new techniques for the study of e-generalization problems, for equational theories in general and (algebraizable) logics in particular. We have highlighted the invariants up to generality of e-generalization problems and their solutions, using the congruences of the 1-generated free algebras; these congruences always shed some light on the poset of solutions of a problem, and in some cases they fully characterize it. 

As hinted at in the last section, our methods can be used fruitfully to study a plethora of different examples. Among other varieties that should be amenable to our methods, a most interesting example is given by Heyting algebras, or, from the logical perspective, intuitionistic logic. Indeed, Heyting algebras are a 1EP variety. To tackle this and related examples one would require an in depth study that goes beyond the boundaries of this work; nonetheless we plan to pursue it in future work. 

Another relevant direction to follow is the role played by e-generalization in {\em refutation systems}, i.e., axiomatic systems meant to derive non-valid formulas \cite{Skura1,Skura2,Citkin}; in such systems the {\em reverse substitution rule} derives from a non-valid formula the refutation of all its generalizers. A connection with the present theory should be investigated.

 Furthermore, it would be interesting to extend the study of the examples considered in this work to the corresponding varieties with infinitely many constants in the signature; this in general will be non-trivial. Indeed, we observe that the main difference with respect to universal-algebraic operations is given by subalgebras, which would need to include such new constants. This entails that the classes of projective and exact algebras, based on the existence of particular embeddings, will generally differ from the ones studied here.
 
Finally, generalization has found many applications in several domains, such as inductive logic programming \cite{Muggleton,CropperMorel,Gulwani}, parallel programming \cite{Barwell},  conceptual blending \cite{Eppeetal}, case-based reasoning \cite{ArmengolPlaza}; addressing the connections with our theoretical setting could be wortwhile.


\end{document}